\documentclass[11pt,           
english,     
]{article}

\usepackage{amsmath}  
\usepackage{amssymb}         
\usepackage[utf8]{inputenc}   
\usepackage[T1]{fontenc}      
\usepackage[sort]{cite}        
\usepackage{fancyhdr}          
\usepackage[a4paper]{geometry} 
\usepackage{xspace} 
\usepackage{tikz-cd}           
\usepackage{tikz}          
\usetikzlibrary{cd}
\usetikzlibrary{arrows.meta} 
\usetikzlibrary{arrows,positioning}

\usepackage{subcaption}

\newcommand{\doi}[1]{\url{https:/doi.org/#1}}

\usepackage{nchairx} 
\usepackage[expansion=false]{microtype}        
\usepackage[%
final=true,         
pdfpagelabels,		
hypertexnames=false 
]{hyperref}

\makeatletter

\newcommand{\bibnote}[2]{\nocite{#1}\@namedef{#1chairxnote}{#2}}
\makeatother

\renewcommand{\theta}{\vartheta}
\renewcommand{\epsilon}{\varepsilon}
\renewcommand{\vec}[1]{\boldsymbol{#1}}

\newcommand{\cat}[1]{\mathbf{#1}}
\newcommand{\limit}{\underleftarrow{\mathrm{lim}}\;}
\newcommand{\colimit}{\underrightarrow{\mathrm{lim}}\;}
\newcommand{\lub}{\lor}
\newcommand{\glb}{\land}
\newcommand{\op}{\mathrm{op}}
\renewcommand{\emph}[1]{\textbf{#1}}
\DeclareMathOperator{\Int}{Int}
\DeclareMathOperator{\rk}{rk}
\DeclareMathOperator{\Z}{\mathbb{Z}}
\DeclareMathOperator{\F}{\mathbb{F}}
\DeclareMathOperator{\R}{\mathbb{R}}
\DeclareMathOperator{\im}{im}

\newcommand{\keywords}[1]{\noindent \textbf{Keywords: {#1}}}

\begin{document}

\title{Spatiotemporal Persistence Landscapes}

\author{Martina Flammer\footnote{Institute of Mathematics, Julius-Maximilians-Universit{\"a}t W{\"u}rzburg, Germany} $^{,}$\footnote{Center for Signal Analysis of Complex Systems, Ansbach University of Applied Sciences, Germany}, Knut H{\"u}per$^\ast$}
\date{\texttt{\{martina.flammer,knut.hueper\}@uni-wuerzburg.de}\\~\\
	\today}

\maketitle

\begin{abstract}
A method to apply and visualize persistent homology of time series is proposed. The method captures persistent features in space and time, in contrast to the existing procedures, where one usually chooses one while keeping the other fixed. An extended zigzag module that is built from a time series is defined. This module combines ideas from zigzag persistent homology and multiparameter persistent homology. Persistence landscapes are defined for the case of extended zigzag modules using a recent generalization of the rank invariant (Kim, M\'{e}moli, 2021). This new invariant is called \textit{spatiotemporal persistence landscapes}. Under certain finiteness assumptions, spatiotemporal persistence landscapes are a family of functions that take values in Lebesgue spaces, endowing the space of persistence landscapes with a distance. Stability of this invariant is shown with respect to an adapted interleaving distance for extended zigzag modules. Being an invariant that takes values in a Banach space, spatiotemporal persistence landscapes can be used for statistical analysis as well as for input to machine learning algorithms.
\end{abstract}


\medskip

\section{Introduction}

Topological Data Analysis (TDA) is a branch of applied mathematics that arose in the early 2000s. The aim of TDA is to utilize topological methods to obtain information from high-dimensional and noisy datasets. So far, it has been successfully applied to a broad variety of real world data, for example biological/biomedical applications \cite{biological1, biological2, biological3, piangerelli2018, atienza2021}, financial data \cite{gidea2017, gidea2018}, dynamical systems \cite{venkataraman2016, maletic2016, mittal2017, jaquette2020} and robotics \cite{robotics1, robotics2}, to name just a few. \\
One of the most important methods in TDA is persistent homology, which is based on the algebraic topological concept of homology. The initial idea behind homology was to characterize shapes by their holes. Mathematically, holes are described by a sequence of \textit{homology groups} that contain the information about the $p$-dimensional holes. The generators of zeroth homology group correspond to the connected components of the topological space, the first homology group describes loops and tunnels, the second voids, and the higher homology groups describe higher dimensional holes. Given point cloud data, persistent homology tracks the homology of that data across several spatial scales and in turn visualizes the information about the parameter values of the birth and death of the homological features in the so-called \textit{barcode}. Features that survive over a large number of scales are said to be \textit{persistent} and correspond to actual features of the data, whereas those features that survive only for a short scale are regarded as noise, making persistent homology robust with respect to small perturbations.\\
A different way to visualize persistent homology is via the \textit{persistence landscapes}. Originally, it was defined for oneparameter persistent homology \cite{bubenik2015}, being extended to multiparameter persistent homology \cite{vipond2020}. In case of oneparameter persistence, it contains the same information as the barcode. The main advantage of oneparameter as well as multiparameter persistence landscapes is that it is a stable vectorization and hence, it allows the calculation of means and variances. Additionally, it is well-suited as input for machine learning tasks. All these advantages led to numerous applications in various fields of data analysis \cite{biological1, corcoran2017, gidea2018, stolz_persistent_2017, atienza2021}.\\
A particular type of input data to persistent homology are time series. In many cases where persistent homology is applied to time series, the pipeline is to first window the time series and subsequently, to build a simplicial complex out of the data in the windows, followed by a comparison of the persistent homology in the distinct windows \cite{perea2019, seversky2016, gidea2018, flammer2022, merelli2016}. In a different approach, some authors examine the use of \textit{zigzag} persistent homology \cite{carlsson2010_zigzag} in order to track the persistence of the topological features throughout the time evolution \cite{tymochko2020, corcoran2017}. However, varying the time parameter, the authors had to choose fixed values for the spatial parameter. To overcome this issue, we propose a method to compute persistence landscapes for a novel kind of bifiltration, which we call a \textit{spatiotemporal} bifiltration. We call the corresponding persistence module \textit{extended zigzag module}. Since this type of bifiltration does not fit into the framework of multiparameter persistent homology, we adapt the notions used in the definition of persistence landscapes to our setting by using the \textit{generalized rank invariant} \cite{kim2021} and provide an algorithm to compute the landscapes. In recent work, the generalized rank invariant has already been used to compute the \textit{generalized rank invariant landscape (GRIL)} \cite{gril2023}. This landscape differs from our landscape by being defined for biparameter persistence modules as well as by the choice of intervals used to compute the landscape.\\
Moreover, we show that spatiotemporal persistence landscapes are stable with respect to the interleaving distance for extended zigzag modules, which we define as an extension of the known interleaving distance for zigzag modules \cite{botnan2018}. In addition, spatiotemporal persistence landscapes can be viewed as random variables that take values in a Banach space and as such, obey the law of large numbers and the central limit theorem. \\
In this article, we both review known results that are necessary for a comprehensive understanding and introduce a new method. Hence, the article is structured as follows. In Section \ref{sec:background}, we give a short overview over the background of persistent homology, its application to time series analysis as well as a brief introduction to the \textit{generalized rank invariant}.  In Section \ref{sec:spatiotemp_landscapes}, we define spatiotemporal persistence landscapes and discuss their properties. We define an interleaving distance for extended zigzag modules and proof stability of the spatiotemporal persistence landscapes with respect to this interleaving distance in Section \ref{sec:stability}. The algorithm is described in Section \ref{sec:algorithm}, followed by applications to simulated data in Section \ref{sec:applications}. We finish with a discussion in Section \ref{sec:discussion}.

\section{Preliminaries}\label{sec:background}
\subsection{Persistent homology}
One of the major tasks in TDA is to determine the structure of point cloud data that was sampled on a space. In order to achieve that, a filtration of a simplicial complex $K_0\subset K_1 \subset ... \subset K_n$ (most frequently the Vietoris-Rips complex \cite{edelsbrunner2010}) is constructed from point cloud data and the $p$-th simplicial homology functor $H_p(-)$ with coefficients in a field $\mathbb{F}$ (usually $\mathbb{F}_2$) is applied to the filtration. This results in a sequence of vector spaces $M_0,\ldots,M_n$ and associated linear maps $M(a\leq b):M_a\to M_b$ for all $0\leq a\leq b\leq n$, the so-called (oneparameter) persistence module. The maps $M(a\leq b)$ are also referred to as \textit{structure maps}. Alternatively, persistence modules can be defined as functors from a partially ordered set (poset) to the category of vector spaces over a fixed field. In the case where the poset is $\mathbb{R}$ or $\mathbb{Z}$, or a subset thereof, we call it a oneparameter persistence module. When the indexing poset is $\mathbb{R}^n$ we call it a multiparameter persistence module. When the target category is the category of finite dimensional vector spaces, the persistence module is called \textit{pointwise finite dimensional}. By $\cat{Vec}$ we mean the category of all vector spaces over a fixed field and by $\cat{vec}$ the full subcategory of finite dimensional vector spaces.\\
According to the \textit{structure theorem} \cite{zomorodian2005}, oneparameter persistence modules can be uniquely decomposed into a direct sum of interval modules, which are defined as follows.
\begin{definition} Let $b\leq d$. An \textbf{interval module} $I_{[b,d]}$ with birth time $b$ and death time $d$ is defined as
	\begin{align*}
		I_i=\begin{cases}
			\mathbb{F} & \mbox{for } b\leq i\leq d, \\
			0 & \mbox{otherwise},
		\end{cases}
	\end{align*}
	where all the maps between $I_i$ and $I_{i+1}$ are identity maps if both $I_i$ and $I_{i+1}$ are $\mathbb{F}$ and zero maps otherwise.
\end{definition}  
The multiset of intervals that appear in the direct sum decomposition is a complete discrete invariant of the persistence module \cite{zomorodian2005}, and is called the \textit{barcode} or \textit{persistence diagram}. Contrarily, in the case of multiparameter persistence modules no complete discrete invariant exists \cite{carlsson2009_multi}. As a result, plenty of research has been going on in defining other invariants for multiparameter persistent homology that have some discriminative properties or are suitable to visualize relevant homological features, like the fibered barcode \cite{lesnick2015fibered}, the rank invariant \cite{carlsson2009_multi}, persistence landscapes \cite{vipond2020}, and many more. The rank invariant sends a pair of indices $a\leq b$ to the rank of the linear map $M(a\leq b)$. For oneparameter persistence modules, the rank invariant and the barcode determine each other and hence, the rank invariant is also a complete discrete invariant.

\subsection{Persistence landscapes}
In \cite{bubenik2015}, Bubenik introduced persistence landscapes, a stable representation of barcodes of oneparameter persistence modules by converting a barcode into a function. This leads to a summary that lies in a function space, which makes it possible to apply techniques from statistical analysis and machine learning. This summary has also been adapted to multiparameter persistent modules \cite{vipond2020}. In the following, we recall the definitions for both cases. 
\subsubsection{Oneparameter persistence landscapes}
Let $M$ be oneparameter persistence module, i.e. a functor $M:\mathbb{R}\to \cat{vec}$.
For any indices $a,b\in\mathbb{R}$, we define the \textit{rank function} $\beta_0: \mathbb{R}^2\rightarrow \mathbb{R}$ as follows:
\begin{align*}
	\beta_0 (a,b):=\begin{cases}
		\dim(\im(M(a\leq b)) ) & \mbox{if } a\leq b,\\
		0 & \mbox{otherwise}.
	\end{cases}
\end{align*}
A change of coordinates $m:=\frac{a+b}{2}$ and $h:=\frac{b-a}{2}$ leads to a function that is supported on the upper half plane instead of being supported above the diagonal. By this rescaling, one changes from coordinates that correspond to births and deaths to coordinates that correspond to \textit{midpoints} and \textit{half-lives} of the features. This function is called the \textit{rescaled rank function} $\beta:\mathbb{R}^2\rightarrow\mathbb{R}$:
\begin{align*}
	\beta(x,h):= \begin{cases}
		\dim(\im(M(x-h\leq x+h))) & \mbox{if } h\geq 0, \\
		0 & \mbox{otherwise.}
	\end{cases}
\end{align*}
\begin{definition}(Persistence landscapes \cite{bubenik2015})\\
	Let $M:\mathbb{R}\to\cat{vec}$ be a oneparameter persistence module. The persistence landscape of $M$ is defined as a sequence of functions $\lambda_k:\mathbb{R}\rightarrow\mathbb{R}\cup\{-\infty,\infty\}$ with
	\begin{align*}
		\lambda_k(x):=\sup\{h \geq 0\; : \; \beta(x,h)\geq k \}.
	\end{align*}
\end{definition}
In other words, $\lambda_k(t)$ is the maximal half-length of an interval being centered at $x$ and is contained in at least $k$ intervals of the barcode \cite{vipond2020}.
In Figure \ref{fig_1pLand}, one can see an example of a persistence diagram and the corresponding persistence landscape.
\begin{remark}\label{rem:kmax}
	A simple way to calculate the persistence landscape is given by the observation in \cite{bubenik2015} that for a persistence diagram $\{(b_i,d_i)\}_{i=1}^n$ the landscape can be determined as
	\begin{align*}
		\lambda_k(x)= \mbox{ k-th largest value of} \max\bigl(\min(x-b_i,d_i-x),0\bigr).
	\end{align*}	
\end{remark}
\begin{remark}
	It is known that the barcode and the oneparameter persistence landscape determine each other and hence, the persistence landscape is also a complete invariant. Since zigzag persistent modules also decompose into a direct sum of interval modules one can define barcodes and hence, persistence landscapes also in the case of zigzag persistence. Analogously to the case of oneparameter persistence modules, we obtain a complete invariant.
\end{remark}

\begin{figure}
	\includegraphics[width=0.49\textwidth]{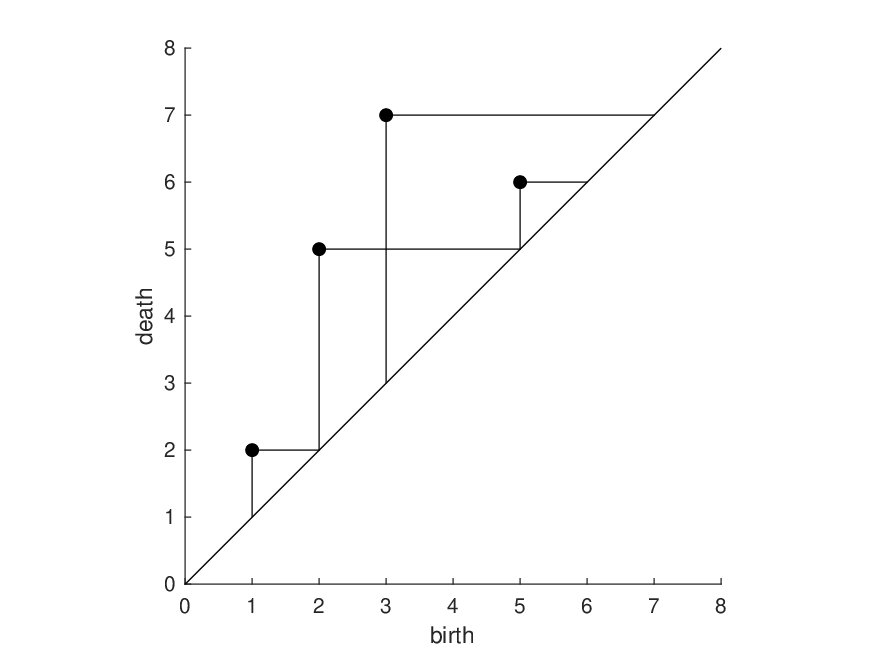}
	\includegraphics[width=0.49\textwidth]{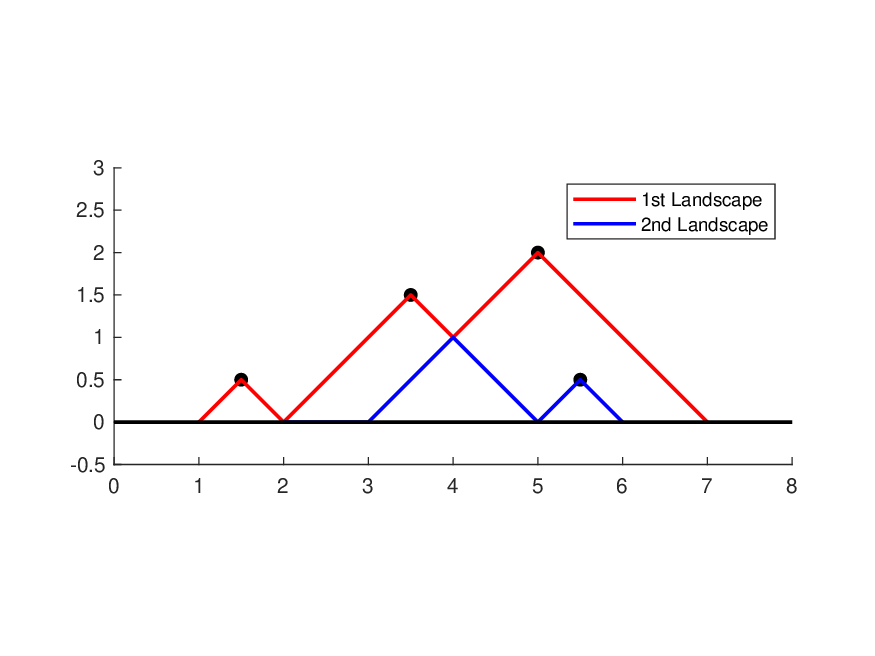}
	\caption{The left hand side shows a persistence diagram (black dots) and the right hand side the corresponding landscapes $\lambda_1$ and $\lambda_2$.}\label{fig_1pLand}
\end{figure}

\subsubsection{Multiparameter persistence landscape}
In \cite{vipond2020}, Vipond generalized the notions of rescaled rank invariant and persistence landscapes to multiparameter persistence modules in a natural way. \\
In the following, let $(\mathbb{R}^n,\leq)$ be the poset defined such that $a\leq b$ if and only if $a_i\leq b_i$ for all $i=1,...,n$.
\begin{definition}
	Let $M$ be a multiparameter persistence module, then the rank function $\beta_0:\mathbb{R}^{2n}\to\mathbb{R}$ of $M$ for $a,b\in\mathbb{R}^n$ is defined as\begin{align*}
		\beta_0(a,b):=\begin{cases}
			\dim(\im(M(a\leq b))) & \mbox{if } a\leq b, \\
			0&\mbox{otherwise.}
		\end{cases}
	\end{align*}
	The rescaled rank function $\beta:\mathbb{R}^{2n}\rightarrow\mathbb{R}$ is defined as
	\begin{align*}
		\beta(x,h):=\begin{cases}
			\dim(\im(M(x-h\leq x+h))) & \mbox{if } h\geq 0,\\
			0&\mbox{otherwise.}
		\end{cases}
	\end{align*}
\end{definition}
\begin{definition}(Multiparameter persistence landscape \cite{vipond2020})
	The multiparameter persistence landscape considers the maximal radius over which $k$ features persist in every (positive) direction through $x$ in the parameter space
	\begin{align*}
		\lambda_k(x):=\sup\{\epsilon\geq 0\; :\; \beta(x,h)\geq k \mbox{ for all } h\geq 0\mbox{ with } \|h\|_\infty\leq \epsilon\}.
	\end{align*}
\end{definition}
Restricting the multiparameter persistence landscape to oneparameter persistence modules gives exactly the definition of a oneparameter persistence landscape. \\
The following lemma from \cite{vipond2020} allows to reduce the computational cost for the calculation of the multiparameter persistence landscape. 
\begin{lemma}\label{lemma:fibered_barcode}
	Let $M$ be a multiparameter persistence module with rank function $\beta_0(\cdot,\cdot)$. Let $1\in\mathbb{R}^n$ be the vector where every entry is $1$. For all $h\geq 0$ we have $\beta_0(x-\|h\|_\infty 1,x+\|h\|_\infty 1)\leq \beta_0(x-h,x+h)$.
\end{lemma}
An immediate consequence is that one only needs to compute $\sup \{\epsilon\geq 0 : \beta_0(x-\epsilon 1,x+\epsilon 1)\geq k \}$ in order to get the value of the multiparameter persistence landscape $\lambda_k$ at point $x$. In other words, the barcode in a diagonal direction contains the information about all landscapes of all points lying on that diagonal. \\
Intuitively, regions in the landscape with large values correspond to features which are robust with respect to changes in the filtration parameters. Furthermore, if for large $k$ the landscape is non-zero it indicates that there is a large number of homological features.

\subsection{Persistent homology for time series}
In many cases where persistent homology is applied to time series, the pipeline is to window the time series and subsequently, to build a simplicial complex out of the data in the windows and compare the persistent homology of the different windows \cite{perea2019, seversky2016, gidea2018, flammer2022, merelli2016}. Frequently, an univariate time series is given and \textit{time delay embedding} is used to construct a point-cloud \cite{perea2019, pereira2015,seversky2016}. However, there are also approaches that utilize \textit{zigzag persistent homology} to analyze time series by tracking the homological changes of point cloud throughout the time evolution \cite{tymochko2020, corcoran2017}. We give a short summary of both approaches in the following.

\subsubsection{Time delay embedding}\label{subsubsec:delay_emb}
In real world applications, frequently not the entire information about a dynamical system is known. Instead, only certain quantities can be measured, which is described by a so-called \textit{observation function}, usually an univariate time series $x=(x_1,...,x_n)$. The delay embedding of one observation function $x$ is defined as
\begin{align*}
	Y=\bigl\{y\in\R^d:y=(x_i,...,x_{i+(d-1)\tau})\bigr\}
\end{align*}
with embedding dimension $d$ and delay parameter $\tau$. According to Takens embedding theorem \cite{takens1981}, the delay embedding of the observation function has the same topological structure as the state space of the dynamical system under certain, not very restrictive, assumptions. 

\subsubsection{Zigzag persistent homology for time series}

In \cite{tymochko2020}, the authors examined a different approach by using zigzag persistent homology in order to track the evolution of the topological features over time.  
Following the procedure in \cite{carlsson2010_zigzag}, they constructed a zigzag sequence of simplicial complexes by including the simplicial complexes of two neighboring point clouds $X_i$ and $X_{i+1}$ ($i=1,\ldots,n-1$) into a bigger space $X_i\cup X_{i+1}$ as follows:
\begin{equation*}
	\begin{tikzcd}[column sep=1ex,ampersand replacement=\&]
		X_0 \arrow[rd, hook] \& \&  X_1 \arrow[dl, hook]\arrow[dr, hook] \& \&  X_2 \arrow[dl, hook] \& \cdots \&  X_{n-1}\arrow[dr, hook] \&\& X_n\arrow[dl, hook] \\
		\& X_0\cup X_1 \& \& X_1\cup X_2 \& \& \& \& X_{n-1}\cup X_n \& 
	\end{tikzcd}
\end{equation*}
In this sequence, generators at time steps $i$ and $i+1$ that generate the same feature in $X_i\cup X_{i+1}$ are said to belong to the same feature, but at different time steps. Hence, zigzag persistent homology tracks the persistence with respect to time instead of spatial persistence. Notice that for multi-variate time series, one possible way to construct a zigzag sequence is to partition the time series into windows and build a simplicial complex on that data. However, since the vertex sets of neighboring windows are disjoint, the union is also disjoint. To avoid this, we build the intermediate step by taking the union of the point clouds and building a Vietoris-Rips complex on the union point cloud.\\
Similarly to oneparameter modules, zigzag modules decompose into a direct sum of interval modules and thus, the barcode (resp. persistence diagram) of a zigzag module is a complete invariant. 
To construct the simplicial complexes $X_i$, the authors of \cite{tymochko2020} used Vietoris Rips complexes at specified radii $\epsilon_i$. However, it is a priori not clear how to choose the radii. Our approach is to combine zigzag filtrations with filtrations in a spatial direction and therefore, to regard diagrams of the following form, where the superscript $\epsilon_i$ denotes the scale of the Vietoris-Rips complex.

\begin{equation*}
	\begin{tikzcd}[ampersand replacement=\&]
		X_0^{\epsilon_0} \arrow[d, hook] \arrow[r, hook] \& (X_0 \cup X_1)^{\epsilon_0} \arrow[d, hook] \& \arrow[l, hook] X_1^{\epsilon_0} \arrow[d, hook]\arrow[r, hook] \&(X_1 \cup X_2)^{\epsilon_0} \arrow[d, hook] \& \arrow[l, hook] X_2^{\epsilon_0}\arrow[d, hook]  \arrow[r, hook] \& ...\\
		X_0^{\epsilon_1}\arrow[d, hook] \arrow[r, hook] \& (X_0 \cup X_1)^{\epsilon_1}\arrow[d, hook] \& \arrow[l, hook] X_1^{\epsilon_1}\arrow[d, hook] \arrow[r, hook] \& (X_1 \cup X_2)^{\epsilon_1}\arrow[d, hook] \& \arrow[l, hook] X_2^{\epsilon_1}\arrow[d, hook]  \arrow[r, hook] \& ...  \\
		X_0^{\epsilon_2} \arrow[d, hook] \arrow[r, hook] \& (X_0 \cup X_1)^{\epsilon_2}\arrow[d, hook] \& \arrow[l, hook]X_1^{\epsilon_2}\arrow[d, hook] \arrow[r, hook]  \& (X_1 \cup X_2)^{\epsilon_2}\arrow[d, hook] \& \arrow[l, hook] X_2^{\epsilon_2} \arrow[d, hook] \arrow[r, hook] \& ...  \\
		\vdots\& \vdots \&\vdots \& \vdots \&\vdots\&
	\end{tikzcd}
\end{equation*}
Since the above diagram is not a multiparameter filtration, none of the common invariants can be applied to it. We call the corresponding persistence module an \textit{extended zigzag module}. In this work, we propose an algorithm to calculate persistence landscapes for these modules by using the generalization of the rank invariant introduced in \cite{kim2021}.

\subsection{Generalized Rank}\label{subsec:gen_rank}
For zigzag modules, we require a generalization of the rank invariant that is equivalent to the barcode, because we seek to obtain information about the persistence of features in time. In \cite{puuska2020}, the author defined the rank invariant for multiparameter persistence modules as the map that sends a tuple of points $(a,b)$, where $a<b$, to the rank of the map $M(a<b)$. However, for zigzag modules only adjacent indices are comparable (i.e. $a<b$ or $b<a$). The following example shows that the rank invariant in \cite{puuska2020} does not contain all the information about the interval decomposition of the zigzag module.
\begin{example}
	Consider the two zigzag modules
	\begin{equation*}
		\begin{tikzcd}[ampersand replacement=\&]
			M\; : \&  0  \&  \arrow["0",l,labels=above] \mathbb{F} \arrow["(1\; 0)",r]  \&   \mathbb{F}^2  \&  \arrow["(0\; 1)",l,labels=above]\mathbb{F}\arrow["0",r] \& 0, \\
			N\; : \&  0  \&  \arrow["0",l,labels=above] \mathbb{F} \arrow["(1\; 1)",r]  \&   \mathbb{F}^2  \&  \arrow[l,"(1\; 1)",labels=above]\mathbb{F}\arrow["0",r] \& 0,
		\end{tikzcd}
	\end{equation*}
	indexed by $\{1,2,3,4,5\}$. They both have the same rank invariant, but since $M=\mathbb{I}_{[2,3]}\oplus\mathbb{I}_{[3,4]}$ and $N=\mathbb{I}_{[2,4]}\oplus\mathbb{I}_{[3,3]}$, they are not isomorphic.
\end{example}

In \cite{kim2021}, the authors proposed a generalized rank invariant for modules indexed over arbitrary posets. The rank invariant is defined for so-called intervals, which are defined as follows.
\begin{definition}
	Let $P$ be a poset. We call a nonempty subset $I$ of $P$ an \emph{interval} of $P$ if for all $p,q\in I$ and $p\leq r\leq q$ it holds that $r\in I$ and $I$ is connected, i.e. for all $p,q\in I$ there is a sequence $p=p_1,\ldots,p_l=q$ of elements in $I$ such that $p_i$ and $p_{i+1}$ are comparable ($p_i\leq p_{i+1}$ or $p_i\geq p_{i+1}$ for all $1\leq i\leq l-1$). 
\end{definition}
For a persistence module $M:P\to \cat{Vec}$ we denote by $M|_I$ its restriction to a subset $I$ of $P$. Furthermore, we denote by $\limit M|_I=(L,(\pi_p:L\to M_p)_{p\in I})$ the limit of $M|_I$ and by $\colimit M|_I=(C,(i_p:M_p\to C)_{p\in I})$ the colimit of $M|_I$. See Appendix \ref{app:cat_defis} for the definitions of limits and colimits.
\begin{remark}
	Note that from the definitions of limits and colimits it follows that for every $p\leq q$ in $I$ it holds that $M(p\leq q) \circ \pi_p=\pi_q$ and $i_q\circ M(p\leq q)=i_p$. This implies that $i_p\circ \pi_p = i_q\circ \pi_q$ for all $p,q\in I$. This map $\psi_{M|_I}:=i_p\circ \pi_p $ is called the canonical limit-to-colimit map.
\end{remark}
\begin{definition}\label{defi_genrank}
	The \emph{generalized rank} over an interval $I$ of a persistence module $M$ is defined as $\rank(M|_I):= \rank(\psi_{M|_I})$, where $\psi_{M|_I}$ is the canonical limit-to-colimit map $i_p\circ \pi_p$ for any $p\in I$.
\end{definition}
\begin{remark}
	For interval decomposable modules, the generalized rank of $M$ over an interval $I$ equals the number of intervals in the direct sum decomposition of $M$ that contain $I$ and hence, it is a complete invariant \cite{kim2021}. 
\end{remark}
\begin{definition}\label{def_genRI}
	Let $\Int(P)$ be the set of intervals of the poset $P$. The \emph{generalized rank invariant} is the map
	\begin{align}
		\rk : \Int(P)\to \mathbb{N}_0, \quad I \mapsto \rank(M|_I)=\rank(\psi_{M|_I}).
	\end{align} 
\end{definition}
\begin{remark}
	When applied to oneparameter persistence modules, the generalized rank invariant coincides with the standard rank invariant. Furthermore, for zigzag modules the generalized rank over an interval $I$ counts the number of intervals $J$ in the direct sum decomposition of the zigzag module containing $I$. In total, the barcode and the generalized rank invariant of the zigzag module contain the same information.
\end{remark}
As the original rank invariant, the generalized rank invariant is order-reversing.
\begin{lemma}\label{lemma_orderReverse}
	For two intervals $I\subset J$ of $P$, it holds that $\rk(I)\geq \rk(J)$.
\end{lemma}
\begin{proof}
	By construction, the limit of $M|_J$ is also a cone of $M|_I$ and the colimit of $M|_J$ is a cocone of $M|_I$. Hence, by the universal properties of limits and colimits there exist unique morphisms $g:\limit M|_J \to \limit M|_I$ and $f:\colimit M|_I\to\colimit M|_J$ such that for all $a\in I$, all $z\in M_a$ and all $y\in\limit M|_J$ it holds $\pi_a^J(y)=\pi_a^I\circ g(y)$ and $i_a^J(z)=f\circ i_a^I(z)$. Hence, for the canonical limit-to-colimit-maps $\psi_{M|_I}$ and $\psi_{M|_J}$ it holds
	\begin{align*}
		\psi_{M|_J}(y)=i_a^J\circ \pi_a^J(y)=f\circ i_a^I\circ \pi_a^I\circ g(y)=f\circ \psi_{M|_I}\circ g(y)
	\end{align*} 
	and hence, $\psi_{M|_J}$ factors through $\psi_{M|_I}$ and so $\rank(\psi_{M|_I})\geq \rank(\psi_{M|_J})$. 
\end{proof}
The following proposition shows how the limit and colimit can actually be constructed for diagrams valued in $\cat{Vec}$. For this, we first need to establish the following notation: for $p,q\in P$, $v_p\in M_p$ and $v_q\in M_q$ we write $v_p\sim v_q$ if $p$ and $q$ are comparable and either $M(p\leq q)(v_p)=v_q$ or $M(q\leq p)(v_q)=v_p$, whichever case is applicable. The following proposition appears among others in \cite{dey2022}. 
\begin{proposition} \label{prop_limColimInVec} Let $M:P\to\cat{Vec}$ be a persistence module. Then: \begin{itemize}
		\item[(i)]  The limit of $M$ is (isomorphic to) the pair $(L,(\pi_p)_{p\in P})$ where
		\begin{align}
			L:=\left\{  (v_p)_{p\in P} \in\prod\limits_{p\in P} M_p : \forall p\leq q \in P, v_p\sim v_q  \right\}
		\end{align}
		and for each $p\in P$, $\pi_p:L\to M_p$ is the canonical projection. We call an element of $L$ a section of $M$.
		\item[(ii)]The colimit of $M$ is (isomorphic to) the pair $(C,(i_p)_{p\in P})$ described as follows: for $p\in P$ let the map $j_p:M_p\hookrightarrow \oplus_{p\in P}M_p$ be the canonical injection. $C$ is the quotient vector space $\left( \oplus_{p\in P}M_p \right)/ T$, where $T$ is the subspace of $\oplus_{p\in P}M_p$ which is generated by vectors of the form $j_p(v_p)-j_q(v_q)$ for $v_p\sim v_q$ and the maps $i_p:M_p\to C$ are the compositions $\rho\circ j_p$, where $\rho$ is the quotient map $\rho:\oplus_{p\in P}M_p\to C$.
	\end{itemize}
\end{proposition}
\begin{definition}
	A \emph{path} in $P$ between $p\in P$ and $q\in P$ is a sequence of elements $p=p_1,\ldots,p_l=q$ in $P$ such that for all $1\leq i\leq l-1$ the elements $p_i$ and $p_{i+1}$ are comparable. A \emph{section along the path} $\Gamma$ is an $l$-tuple $v\in\oplus_{i=1}^l M_{p_i}$ with the property that $v_{p_i}\sim v_{p_{i+1}}$ for all $i=1,\ldots,l-1$. 
\end{definition}
The following proposition will be utilized in the proof of Theorem \ref{theo_main}. We use the same notation as in Proposition \ref{prop_limColimInVec}.
\begin{proposition}\label{prop_path} \cite{dey2022} For $p,q\in P$ and vectors $v_p\in M_p$ and $v_q\in M_q$, it holds that $[j_p(v_p)]=[j_q(v_q)]\in\colimit M$ if there exists a path $\Gamma$ between $p$ and $q$ and a section $(w_p)_{p\in \Gamma}$ of M along $\Gamma$ such that $w_p=v_p$ and $w_q=v_q$. 
\end{proposition}
\begin{proof}
	It follows directly from the explicit formula for the colimit in Proposition \ref{prop_limColimInVec}.
\end{proof}
To keep notations simple, we write $[v_p]$ instead of $[j_p(v_p)]$ for elements in the colimit.
\begin{remark}
	Note that the converse of Proposition \ref{prop_path} does not hold in general. Consider for example the module
	\begin{equation*}
		\begin{tikzcd}[ampersand replacement=\&]
			\F_2 \arrow[r,"{0}\choose{1}"]\arrow[d, "\id"] \& \F_2^2  \\
			\F_2 \& \arrow[l, "\mathrm{id}"] \F_2 \arrow[u,"{1}\choose{0}"] 
		\end{tikzcd}.
	\end{equation*}
	Passing to the colimit, in the upper right vector space $\left[{{1}\choose{0}}\right] =\left[{{0}\choose{1}} \right]$ holds, because there exists a section along the closed path that identifies the two vectors, and hence $\left[{{1}\choose{1}}\right] =\left[{{0}\choose{0}} \right]$. On the other hand, the vector ${1}\choose{1}$$ \in \F_2^2$ in the upper right vector space does not lie in the image of any of the arrows and thus, there exists no section along any path such that the above holds. 
\end{remark}

The following definition of lower and upper fences will be useful in several places in this article. In particular, these are subsets that carry the important property that they already completely determine the limits and colimits of a diagram.
\begin{definition}
	A subposet $L\subset P$ is called a lower fence of $P$ if $L$ is connected and for any $q\in P$ the intersection $L\cap \{r\in P: r\leq q\}$ is nonempty and connected.\\
	Analogously, a subposet $U\subset P$ is called an upper fence of $P$ if $U$ is connected and for any $q\in P$ the intersection $U\cap \{r\in P:q\leq r\}$ is nonempty and connected.
\end{definition}
The following Proposition already appears among others in \cite{dey2022} and will be used in the proof of the main theorem of this Section.
\begin{proposition}\label{prop_iso_limcolim}
	Let $L$ and $U$ be lower and upper fences, respectively. Given any $P$-indexed persistence module $M$, we have that $\limit M\cong \limit M|_L $ and $\colimit M\cong \colimit M|_U$. 
\end{proposition}
The isomorphisms in Proposition \ref{prop_iso_limcolim} are given by the canonical section extension map 
\begin{align}
	e: \limit M|_L\to \limit M, \quad (v_p)_{p\in L}\mapsto (w_q)_{q\in P},
\end{align}
where for all $q\in P$, $w_q$ is defined as $M(p\leq q )(v_p)$ for any $p\in L\cap\{r\in P: r\leq q\}$ and by the map
\begin{align}
	h:\colimit M|_U\to \colimit M,\quad [v_p]\mapsto [v_p] \quad\forall p\in U,\; v_p\in M_p .
\end{align}
Note that the canonical section extension map $e$ is well-defined since $L\cap\{r\in P: r\leq q\}$ is a connected set. The inverse $r=e^{-1}$ is the canonical section restriction. Furthermore, the other isomorphism $h$ is well-defined because of Proposition \ref{prop_path}. 
Keeping the maps that we defined in the proof of the last Proposition, we set $\xi=h^{-1}\circ \psi_M\circ e$, i.e. we obtain the commutative diagram
\begin{equation}
	\begin{tikzcd}[ampersand replacement=\&]
		\limit M|_L\arrow[r, "\xi"]\arrow[d,"e"] \& \colimit M|_U\arrow[d,"h"]\\
		\limit M  \arrow[r,"\psi_M"] \& \colimit M.
	\end{tikzcd}\label{commDiag1}
\end{equation}
Due to the fact that $e$ and $h$ are isomorphisms we have $\rank \; \xi = \rank\; \psi_M$. In total, since the limit of a diagram is isomorphic to the limit of the diagram restricted to a lower fence and analogously for the colimit and a upper fence, we can compute the generalized rank of a diagram by only calculating limits and colimits of lower and upper fences and taking the canonical map between them. \\
\\
For an interval $I$, denote by $\min (I)$ and $\max (I)$ the set of minimal and maximal elements, respectively, i.e.
\begin{align}
	\min(I)=\{p\in I: \mbox{ there is no } q\in I \mbox{ s.t. }q<p\},\\
	\max(I)=\{p\in I: \mbox{ there is no } q\in I \mbox{ s.t. }p<q\}.
\end{align}

The least upper bound and the greatest lower bound of two elements $p,q\in P$ are denoted by $p\lub q$ and $p\glb q$, respectively.\\
In our work, we will regard persistence modules that are indexed by a subset of $\R^2$, which, however, do not have the usual partial order as $\R^2$. We will specify the partial order in the next section. We sort the elements of $\min(I)$ and $\max(I)$ in ascending order by their $x$-coordinates and the elements of $\max(I)$ in descending order, i.e. $\min(I)=\{ p_0,p_1,\ldots,p_k \}$ and $\max(I)=\{q_0,q_1,\ldots,q_l\}$. Then, we define the two paths
\begin{align}
	\Gamma_{\min}&:p_0,(p_0\lub p_1), p_1 ,(p_1\lub p_2), \ldots ,(p_{k-1}\lub p_k),p_k,\\
	\Gamma_{\max}&:q_0,(q_0\glb q_1),q_1,(q_1\glb q_2),\ldots ,(q_{l-1}\glb q_l),q_l.
\end{align}
Clearly, the set of elements in $\Gamma_{\min}$ is a lower fence of $M$ and the set of elements in $\Gamma_{\max}$ is an upper fence of $M$.
\begin{definition}
	We define the path $\Gamma_{\partial I}$ of an interval $I\subset P$ as the path obtained by composing $\Gamma_{\min}$, any arbitrary path $\Gamma'$ between $p_k$ and $q_l$ (or $p_0$ and $q_0$) and $\Gamma_{\max}$. Further, for a persistence module $M$ we denote by $M_{\partial I}$ its restriction to the path, i.e. $M|_{\Gamma_{\partial I}}$.
\end{definition}
The following theorem is a slight variation of Theorem 24 in \cite{dey2022} and hence, the proof differs only slightly from the proof given in \cite{dey2022}. The difference between our theorem and the theorem in \cite{dey2022} is that we proved it for more general paths $\Gamma_{\partial I}$. The entire proof can be found in Appendix \ref{app:proof_main}.
\begin{theorem}\label{theo_main}
	Let $I\subset P$ be an interval. Then, $\rank(M|_I) = \rank(M_{\partial I})$. 
\end{theorem}
\begin{remark}
	In other words, this theorem allows us to compute the rank over an interval by computing the rank of a zigzag module that is somehow related to the boundary of that interval. As a consequence, it is sufficient to compute ranks of zigzag modules in order to compute generalized ranks over intervals. Hence, any algorithm that computes barcodes of zigzag modules is suitable for the computation of generalized ranks.
\end{remark}
\begin{remark}
	Note that in the proof of the last theorem we did not use the special construction of $\Gamma_{\partial I}$ as the concatenation of $\Gamma_{\min}$ and $\Gamma_{\max}$. In fact, only the properties of lower and upper fences were used. Hence, Theorem \ref{theo_main} holds for any path $\Gamma$ that is composed of a path through a lower fence, any arbitrary path between the lower and the upper fence and a path through an upper fence. 
\end{remark}

\section{Spatiotemporal persistence landscapes}\label{sec:spatiotemp_landscapes}
\subsection{Definition}
We define the underlying poset of the extended zigzag module, which we denote by $(ZZ\times\Z,\ll)$, as follows: $ZZ:= \Z$ as a set and 
\begin{align*}
	(a,b) \ll (a',b') \Leftrightarrow b\leq b' \mbox{ and } \begin{cases}
		a = a'-1 & \mbox{ for } a=2z+1 \mbox{ for some } z\in\mathbb{Z},\\
		a= a'+1 & \mbox{ else}.
	\end{cases}
\end{align*}
We equip the set $ZZ\times\Z$ with the maximum metric $d_m$ as in $\mathbb{Z}^2$, i.e. $d_m(x,y)=\max\{|x_1-y_1|, |x_2-y_2|\}$. Furthermore, we define regions $R_x^\epsilon$ in the parameter space around a point $x\in ZZ\times\Z$ as balls around $x$ with radius $\epsilon$ with respect to the maximum norm, so $R_x^\epsilon=\{y\in ZZ\times\Z:y=x+h \mbox{ with } h\in ZZ\times\Z,\; d_m(h,0)\leq \epsilon \}$. Analogously to \cite{vipond2020}, but adapted to our case of a discrete poset, we define persistence landscapes.
\begin{definition}\label{def:landsc}
	The $k$-th persistence landscape $\lambda_k$ of a persistence module $M:ZZ\times\Z\to\cat{vec}$ considers the maximal radius over which $k$ features persist in every (positive) direction through $x$ in the parameter space
	\begin{align*}
		\lambda_k(x):=\sup\{\epsilon\geq 0 :\; \rank( M|_{R_x^\epsilon})\geq k\}.
	\end{align*}
	The persistence landscape $\lambda$ of $M$ is the map $\lambda:\mathbb{N}\times ZZ\times\Z\to \overline{\mathbb{R}}, \;(k,x)\mapsto \lambda_k(x)$. 
\end{definition}
We want to regard landscapes as functions taking values in $\mathbb{R}$, not as in the definition in the extended real numbers $\overline{\mathbb{R}}$. To assure this, in the following we exclude infinite indecomposables in our persistence module $M$.
\begin{remark}\label{rem:rectangular_regions}
	In this work, we restrict our attention to quadratic regions $R_x^\epsilon$. However, choosing the region we implicitly chose a weight on the spatial and temporal direction. To be precise, the dimensions in space and time are treated equally. In the case where one is interested in treating them differently one could simply consider rectangular regions instead of quadratic regions. The definitions and algorithms can be adapted to this case in a straight forward way. \\
	In principle, any shape of intervals can be chosen for the respective regions. Further possibilities are balls with respect to the Euklidean norm or any other suitable norm. In a different approach, long diagonal regions of certain widths were chosen, which the authors called \textit{worms} \cite{gril2023}.
\end{remark}
\begin{lemma}
	The persistence landscapes have the properties:
	\begin{enumerate}
		\item $\lambda_k(x)\geq 0$,
		\item $\lambda_k(x)\geq \lambda_{k+1}(x)$,
		\item $\lambda_k$ is 1-Lipschitz.
	\end{enumerate}
\end{lemma}
\begin{proof} The first two properties follow directly from the definition.
	For the third, we want to show that $|\lambda_k(y)-\lambda_k(x)|\leq d_m(y,x)$ for all $x,y\in ZZ\times\Z$. Without loss of generality, we assume that $\lambda_k(y)\geq\lambda_k(x)\geq 0$. If $\lambda_k(y)\leq d_m(y,x)$, then $\lambda_k(y)-\lambda_k(x)\leq \lambda_k(y)\leq d_m(y,x)$ and we are done. Thus, we assume that $\lambda_k(y)>d_m(y,x)$. Consider $(h,h)\in ZZ\times\Z$, where $h=\lambda_k(y)-d_m(y,x)$. If follows, that $\lambda_k(y)-h=d_m(y,x)$, and hence, that $y-(\lambda_k(y),\lambda_k(y))\leq x-(h,h)$ as well as $x+(h,h)\leq y+(\lambda_k(y),\lambda_k(y))$. In total,
	\begin{align*}
		y-(\lambda_k(y),\lambda_k(y))\leq x-(h,h) \leq x+(h,h)\leq y+(\lambda_k(y),\lambda_k(y))
	\end{align*}
	and so $R_x^h\subset R_y^{\lambda_k(y)}$ and Lemma \ref{lemma_orderReverse} applies, so $\rank\;R_y^{\lambda_k(y)}\leq \rank\; R_x^h$. With $\lambda_k(x)\geq h=\lambda_k(y)-d_m(y,x)$ the result follows.
\end{proof}
One might ask for the connection between landscapes of an extended zigzag module $M$ and the persistence landscapes of restrictions of $M$ along lines in the zigzag or homogeneously filtered direction. The following proposition answers this question and follows directly from the respective definitions.
\begin{proposition}
	Let $M:ZZ\times\Z\to\cat{vec}$ and $l\subset ZZ\times\Z$ be a line in the horizontal or vertical direction, i.e. $l=\{(a,b)|a\in \mathbb{Z}\}$ for some $b\in\mathbb{Z}$ or $l=\{(a,b)|b\in\mathbb{Z}\}$ for some $a\in\mathbb{Z}$. Let $\lambda(M|_l)$ be the landscape of the oneparameter persistence module or the zigzag persistence module, whichever applies. Then,
	\begin{align*}
		\lambda(M)\leq \lambda(M|_l).
	\end{align*}
\end{proposition}
Under some finiteness assumptions the defined spatiotemporal persistence landscapes can be viewed as being elements of the Banach spaces $L_p(\mathbb{N}\times\mathbb{Z}^2)$. To achieve this, we can restrict to the case where the persistence modules are defined on a bounded set in $ZZ\times \mathbb{Z}$. For applications, this is a reasonable assumption since only finitely many values for the spatial parameter $\epsilon_1,\ldots,\epsilon_n$ are chosen and all time series are finite. Consequently, the values of the landscapes are finite. Being elements of a Banach space, on the space of persistence landscapes we have a notion of distance.
\begin{definition}
	Let $M$ and $N$ be extended zigzag modules such that the respective landscapes are elements of $L^p(\mathbb{N}\times\Z^2)$. The $p$-landscape distance $d_\lambda^p$ is defined as
	\begin{align*}
		d_\lambda^p(M,N)= \| \lambda(M)-\lambda(N)\|_p.
	\end{align*} 
\end{definition}
\begin{example}
	Notice that the landscape distance is only a pseudo-distance on the space of isomorphism classes of persistence modules. For example, the two modules $M=I_{R_{(2,2)}^1}\oplus I_{R_{(2,3)}^1} $ and $N=I_{R_{(2,2)}^1\cup R_{(2,3)}^1 }$ have the same persistence landscape, however, they are clearly not isomorphic. As an immediate consequence, the spatiotemporal persistence landscapes are not a complete invariant.
\end{example}

\subsection{Statistics}
In the setting of oneparameter persistent homology, one major advantage of the use of persistence landscapes over persistence diagrams is that in contrast to persistence diagrams, persistence landscapes allow for a unique mean \cite{bubenik2015}. Analogously to the case of one- and multiparameter landscapes, for spatiotemporal persistence landscapes we can also define a mean landscape of a set of landscapes by taking the pointwise mean. The resulting landscape is in general not the landscape of a persistence module, however, local maxima of the mean landscapes can be interpreted as parameter values where persistent topological features in space and time are located. Furthermore, taking the average landscape over a set of noisy measurements could reduce the influence of noise. \\
Under suitable finiteness assumptions the landscapes are elements of a Banach space, namely the Lebesgue space $L^p(\mathbb{N}\times\mathbb{Z}^2)$ equipped with the usual $p$-norm. We apply the theory of probability in Banach spaces in order to obtain statistical results. In order to guarantee separability of the Lebesgue space we assume that $1\leq p<\infty$. For more details on that topic, see \cite{ledoux1991} and Appendix \ref{app:prob_banach}.

Following the procedures for oneparameter and multiparameter persistence landscapes (\cite{bubenik2015, vipond2020}), we view the spatiotemporal persistence landscapes as random variables that take values in a Banach space. To be precise, let $X$ be a random variable on the probability space $(\Omega,\mathcal{F},P)$, i.e. $X(\omega)$ is the data for $\omega\in\Omega$ with corresponding landscape $\Lambda(\omega)=\lambda(X(\omega))$. Thus, $\Lambda:(\Omega,\mathcal{F},P)\to L^p(\mathbb{N}\times\mathbb{Z}^2)$ is a random variable with values in a Banach space. We denote the expectation value of a real random variable $X$ by $E(X)$. The analogue to the expactation values in case of a random variable $V$ with values in a Banach space is the so-called Pettis integral (see Definition \ref{def:pettis_integral}) and is also denoted by $E(V)$.  \\
We assume that $X_i$ are independent identically distributed copies of $X$ with corresponding landscapes $\Lambda_i$. By $\overline{\Lambda}^n$ we denote the pointwise mean of the first $n$ landscapes. Analogously to the case of oneparameter and multiparameter persistence landscapes, applying the theory of random variables with values in a Banach space yields the following results. These results are stated without proofs because these theorems follow directly from the developed theory for multiparameter persistence landscapes \cite{vipond2020}. This comes from the fact that spatiotemporal persistence landscapes and biparameter persistence landscapes take values in the same Banach space even though we define them for different kind of data.
\begin{theorem}(Strong law of large numbers for spatiotemporal persistence landscapes)\\
	$\overline{\Lambda}^n\to E(\Lambda)$ almost surely if and only if $E(\|\Lambda\|)<\infty$.
\end{theorem}

\begin{theorem}(Central limit theorem for spatiotemporal persistence landscapes)\\
	Let $p\geq 2$, $E(\|\Lambda\|)<\infty$ and $E(\|\Lambda\|^2)< \infty$. Then $\sqrt{n}(\overline{\Lambda}^n-E(\Lambda))$ converges weakly to a Gaussian random variable  with the same covariance structure as $\Lambda$.  
\end{theorem}

\section{Stability of Spatiotemporal Persistence Landscapes}\label{sec:stability}

In this section, we show the stability of spatiotemporal persistence landscapes. Stability is a central part of persistent homology to assure that small perturbations of the input data do not alter the invariant too much. \\
In \cite{vipond2020}, it was shown that multiparameter persistence landscapes are stable with respect to the interleaving distance. Hence, in this section we start by defining an analogue to the interleaving distance for extended zigzag modules, followed by the proof of stability.

\subsection{Block extension functor for zigzag modules}
To define an interleaving distance on extended zigzag modules we extend the approach in \cite{botnan2018} and in \cite{kim2021arXivV4} (version 4 on arXiv). In both approaches, the authors send a zigzag persistence module to a $\R^\op\times\R$-indexed module which allows them to define an interleaving distance for zigzag modules. Here, by $\R^\op$ we mean the opposite category. The partial order on $\mathbb{R}^{\op}\times \mathbb{R}$ (or $\mathbb{Z}^{\op}\times\mathbb{Z}$, respectively) is given by $(a,b)\leq (c,d)$ iff $c\leq a\leq b\leq d$. This can be motivated by the partial order on intervals, where $[a,b]\subset [c,d]$ iff $c\leq a\leq b\leq d$. We slightly change and extend this approach to send an extended zigzag module to an $\Z^\op\times\Z\times\Z$-indexed module in order to define the interleaving distance.\\
At first, we show how to extend a zigzag module to a $\Z^\op\times\Z$-indexed module. 
By zigzag module we mean a functor from the poset $ZZ$ to $\cat{vec}$.
We include $ZZ$ into the poset $\Z^\op\times \Z$ as follows: we map a sink index $i$ to $(i,i)$ and a source index $j$ to $(j+1,j-1)$, shown in Figure \ref{fig:inclusion_and_U}. Note that we required the maps being strictly alternating so that every index is either a sink or a source index.
\begin{figure}
	\centering
	\includegraphics[width=0.42\textwidth]{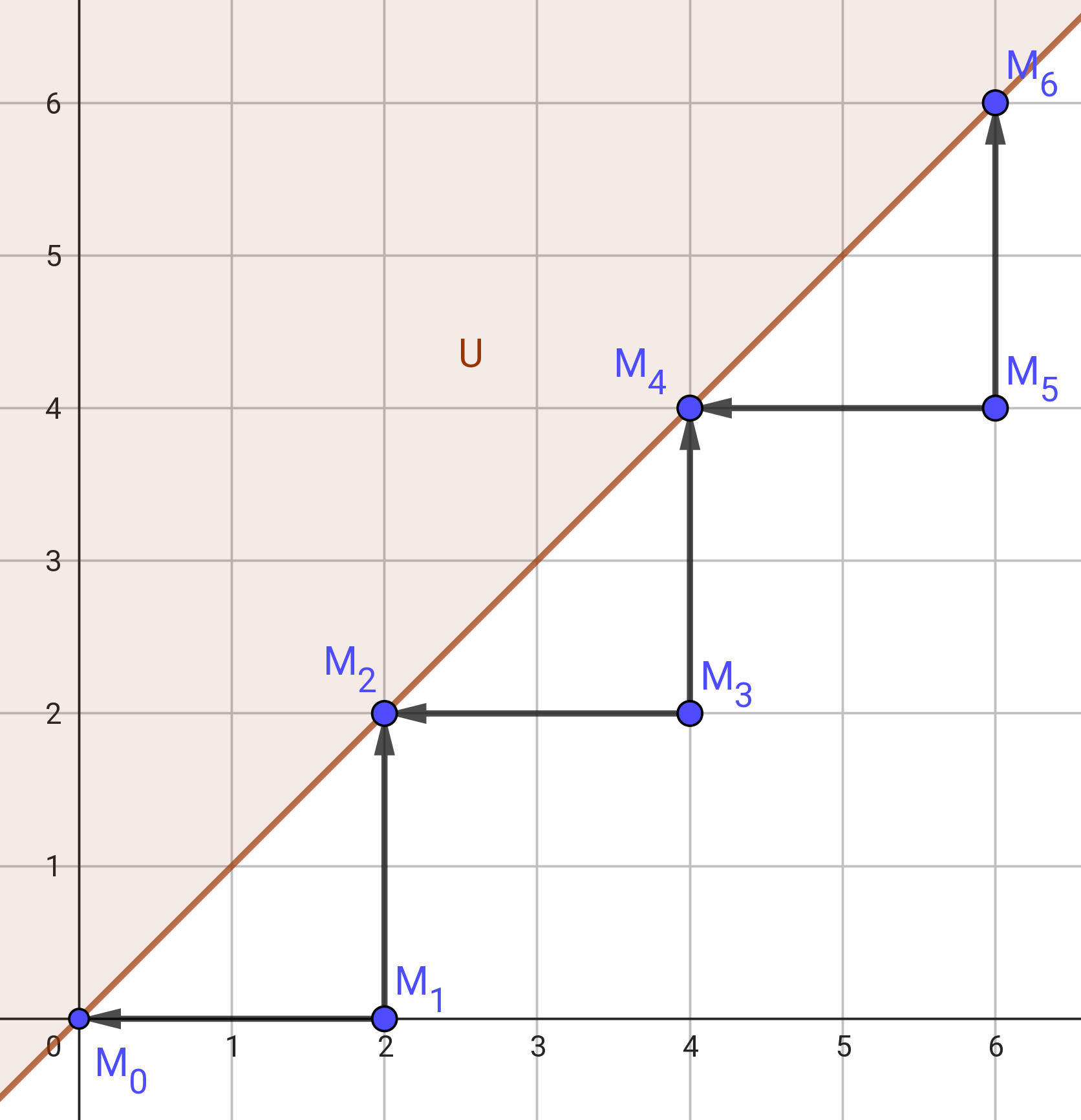}
	\caption{Inclusion of the zigzag poset into $\Z^\op\times\Z$.}\label{fig:inclusion_and_U}
\end{figure}
Since this is an order-preserving map the inclusion $\iota:ZZ\hookrightarrow\Z^\op\times\Z$ is a functor. To map a zigzag module to a $\Z^\op\times\Z$ indexed module we consider the composition of three functors: first, we define $E_1:\cat{vec}^{ZZ}\to\cat{vec}^{\Z^\op\times\Z}$ as the left Kan extension (see Appendix \ref{app:subsec_Kan}) along the inclusion functor $\iota$. Second, we restrict the module to the set $U:=\{(i,j)\in\Z^\op\times\Z|i\leq j\}$ by the restriction functor $(-)|_U:\cat{vec}^{\Z^\op\times\Z}\to\cat{vec}^U$. Finally, we define $E_2:\cat{vec}^U\to\cat{vec}^{\Z^\op\times\Z}$ as the right Kan extension along the canonical inclusion $\kappa:U\hookrightarrow \Z^\op\times\Z$. In total, we define the \textit{block extension} functor
\begin{align*}
	E:=E_2\circ (-)|_U\circ E_1:\cat{vec}^{ZZ}\to\cat{vec}^{\Z^\op\times\Z}.
\end{align*}
It is known that any zigzag module decomposes into a direct sum of interval modules \cite{carlsson2010_zigzag}. Following \cite{botnan2018}, we distinguish four different types of interval modules. Here, we denote by $<$ and $\leq$ the partial order in $\Z^2$.
\begin{align*}
	(a,b)_{ZZ}&:=\{i\in\Z|(a,a)<\iota(i)<(b,b)\} &\mbox{for } a<b\in\Z\cup\{-\infty,\infty\}, \\
	[a,b)_{ZZ}&:=\{i\in\Z|(a,a)\leq\iota(i)<(b,b)\} &\mbox{for } a<b\in\Z\cup\{\infty\}, \\
	(a,b]_{ZZ}&:=\{i\in\Z|(a,a)<\iota(i)\leq(b,b)\} &\mbox{for } a<b\in\Z\cup\{-\infty\}, \\
	[a,b]_{ZZ}&:=\{i\in\Z|(a,a)\leq\iota(i)\leq(b,b)\} &\mbox{for } a\leq b\in\Z.
\end{align*}
By $\langle a,b\rangle_{ZZ}$ we denote an interval of any of these four types. \\
In $\cat{vec}^{\Z^\op\times\Z}$, we consider a special class of persistence modules that are called \textit{block decomposable} modules \cite{botnan2018}. These are modules that decompose into a direct sum of block intervals, where the blocks are sets of the following forms
\begin{align*}
	(a,b)_{BL}&:=\{(x,y)\in\Z^\op\times\Z|a<x,y<b\} &\mbox{for } a<b\in\Z\cup\{-\infty,\infty\}, \\
	[a,b)_{BL}&:=\{(x,y)\in\Z^\op\times\Z|a\leq y<b\} &\mbox{for } a<b\in\Z\cup\{\infty\}, \\
	(a,b]_{BL}&:=\{(x,y)\in\Z^\op\times\Z|a<x\leq b\} &\mbox{for } a<b\in\Z\cup\{-\infty\}, \\
	[a,b]_{BL}&:=\{(x,y)\in\Z^\op\times\Z|x\leq b,y\geq a\} &\mbox{for } a\leq b\in\Z.
\end{align*}
Again, by $\langle a,b\rangle_{BL}$ we denote a block of any of the above types. 
\begin{remark}
	Note that there is a canonical isomorphism between $\cat{vec}^{\Z^\op\times \Z}$ and $\cat{vec}^{\Z\times\Z}$ induced by the isomorphism $\rho:\Z\times\Z\to\Z^\op\times\Z$ sending each $(a,b)$ to $(-a,b)$. As a result, a $\Z^\op\times\Z$-indexed persistence module in general does not decompose into a direct sum of interval modules, just like $\Z^2$-indexed modules. Hence, block decomposable modules are a proper subset of all $\Z^\op\times\Z$-indexed modules.
\end{remark}
\begin{figure}
	\centering
	\includegraphics[width=0.35\textwidth]{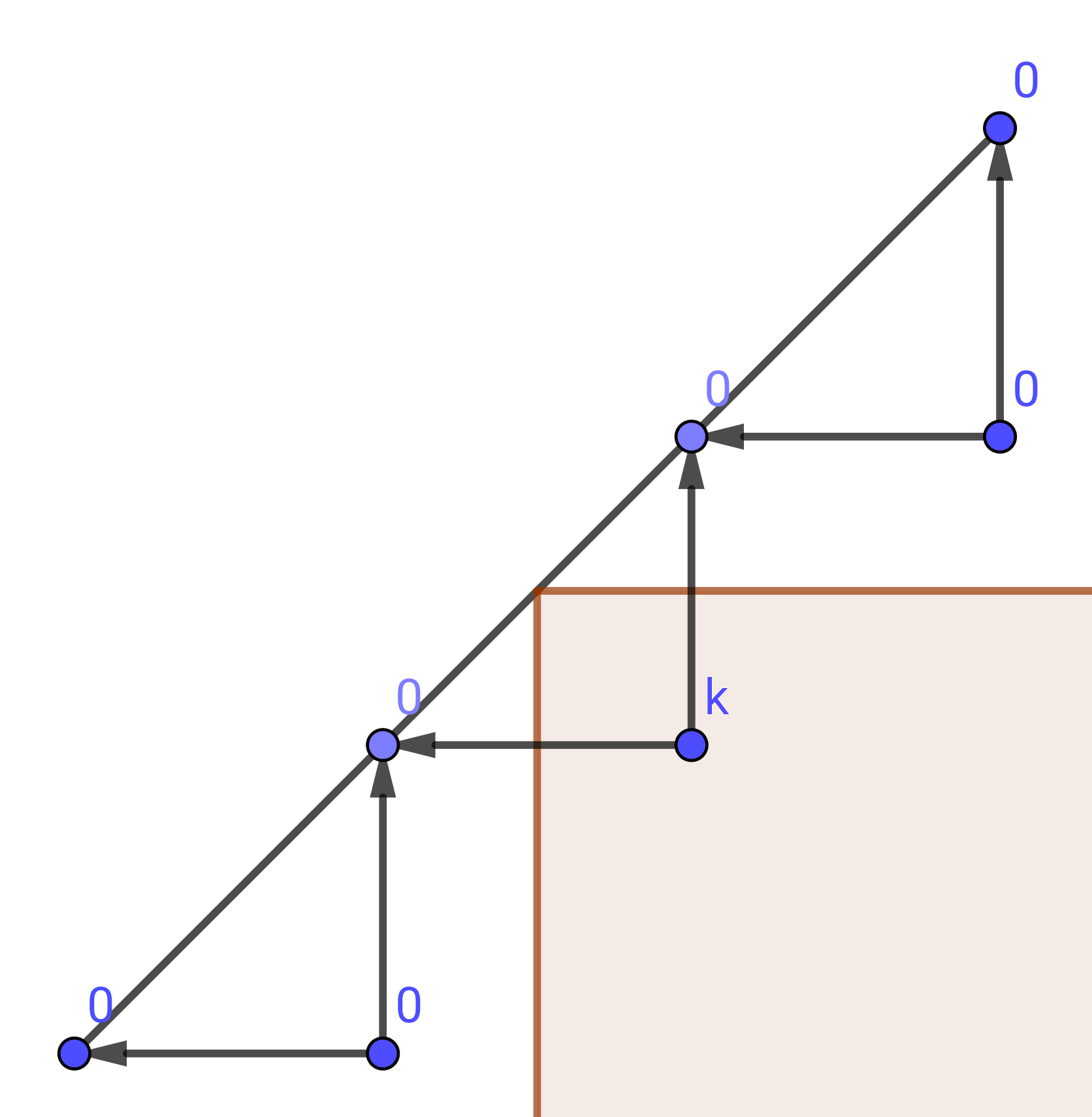}\qquad
	\includegraphics[width=0.35\textwidth]{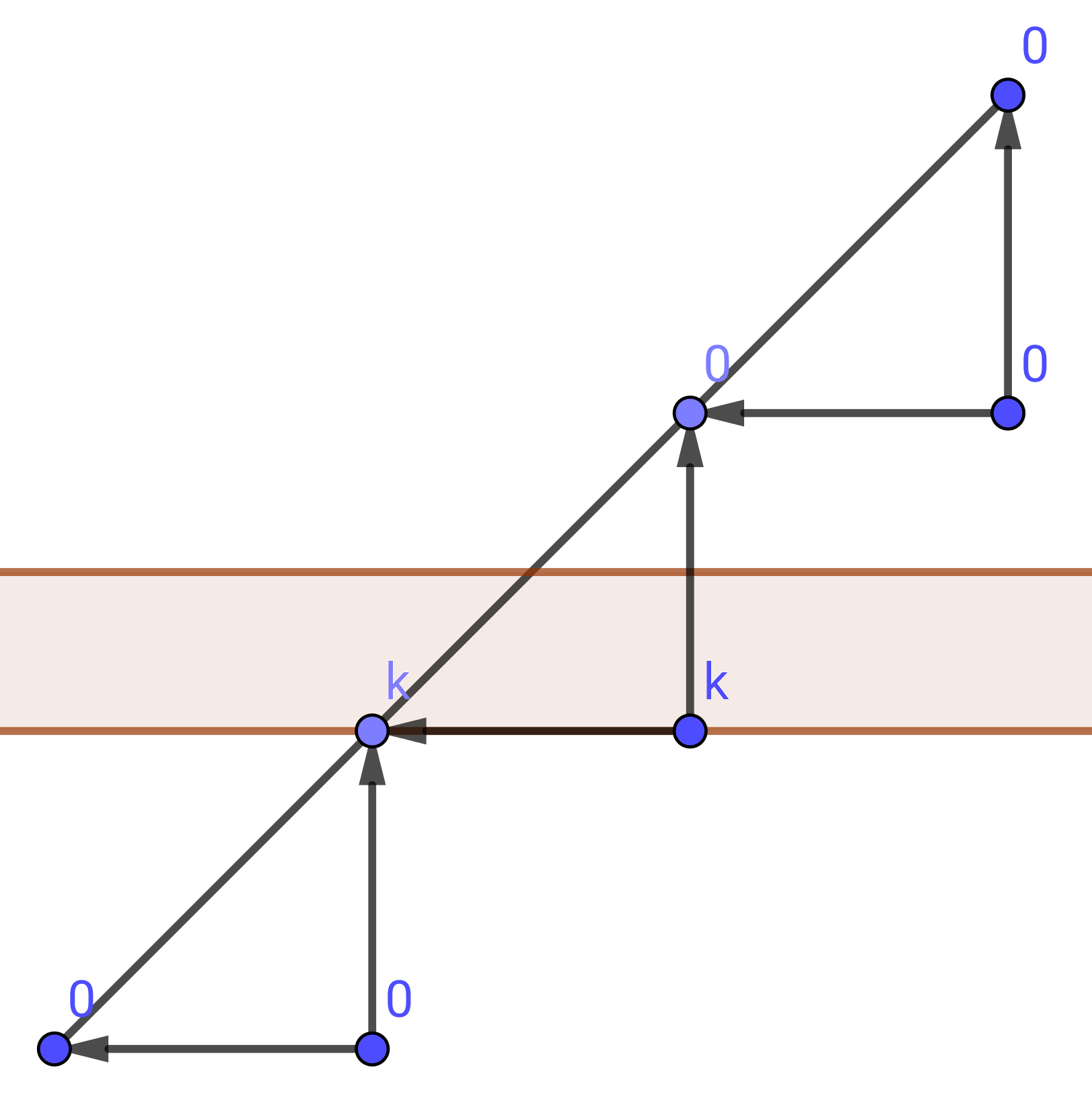}\\
	\vspace{0.35cm}
	\includegraphics[width=0.35\textwidth]{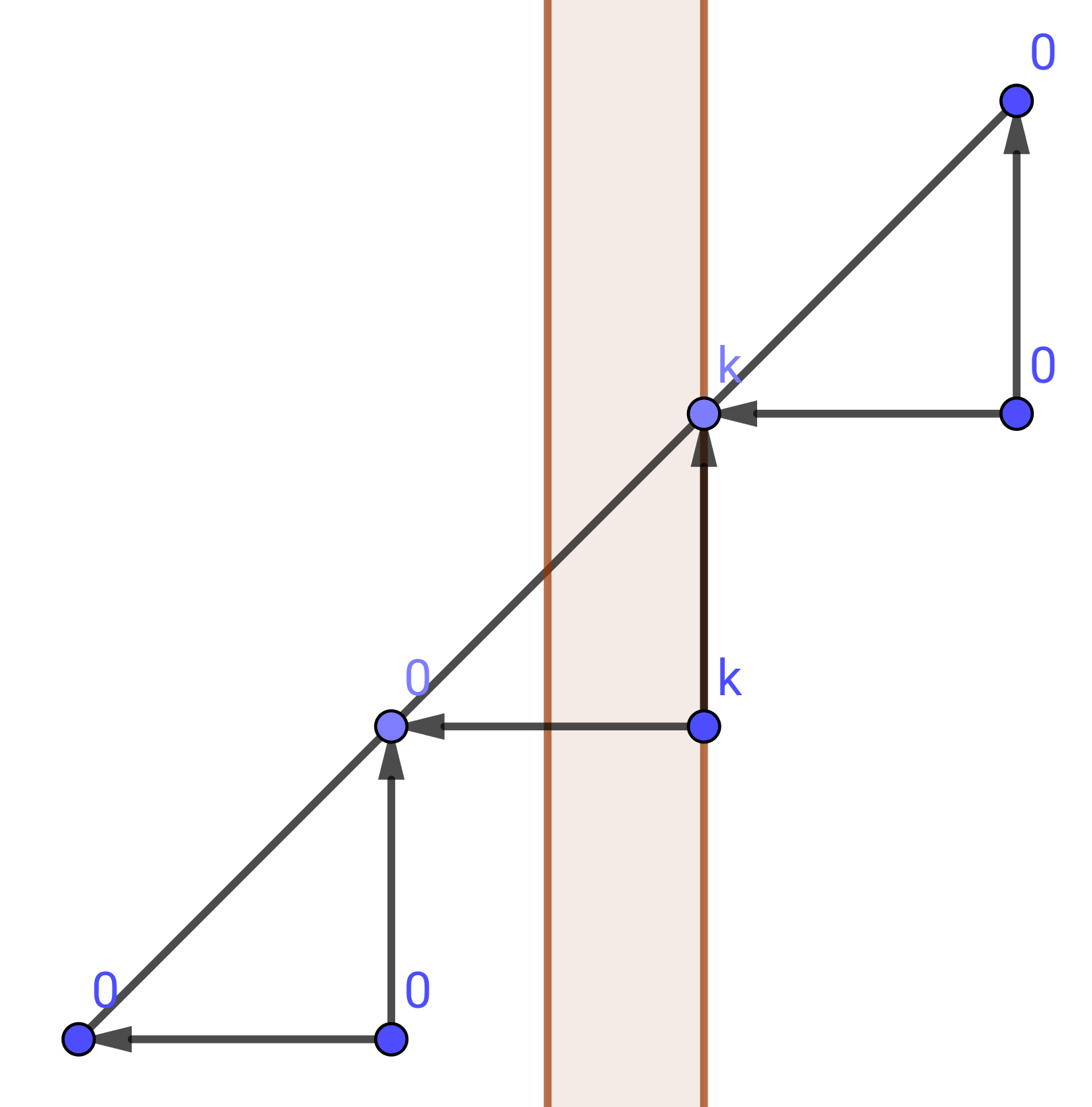}\qquad
	\includegraphics[width=0.35\textwidth]{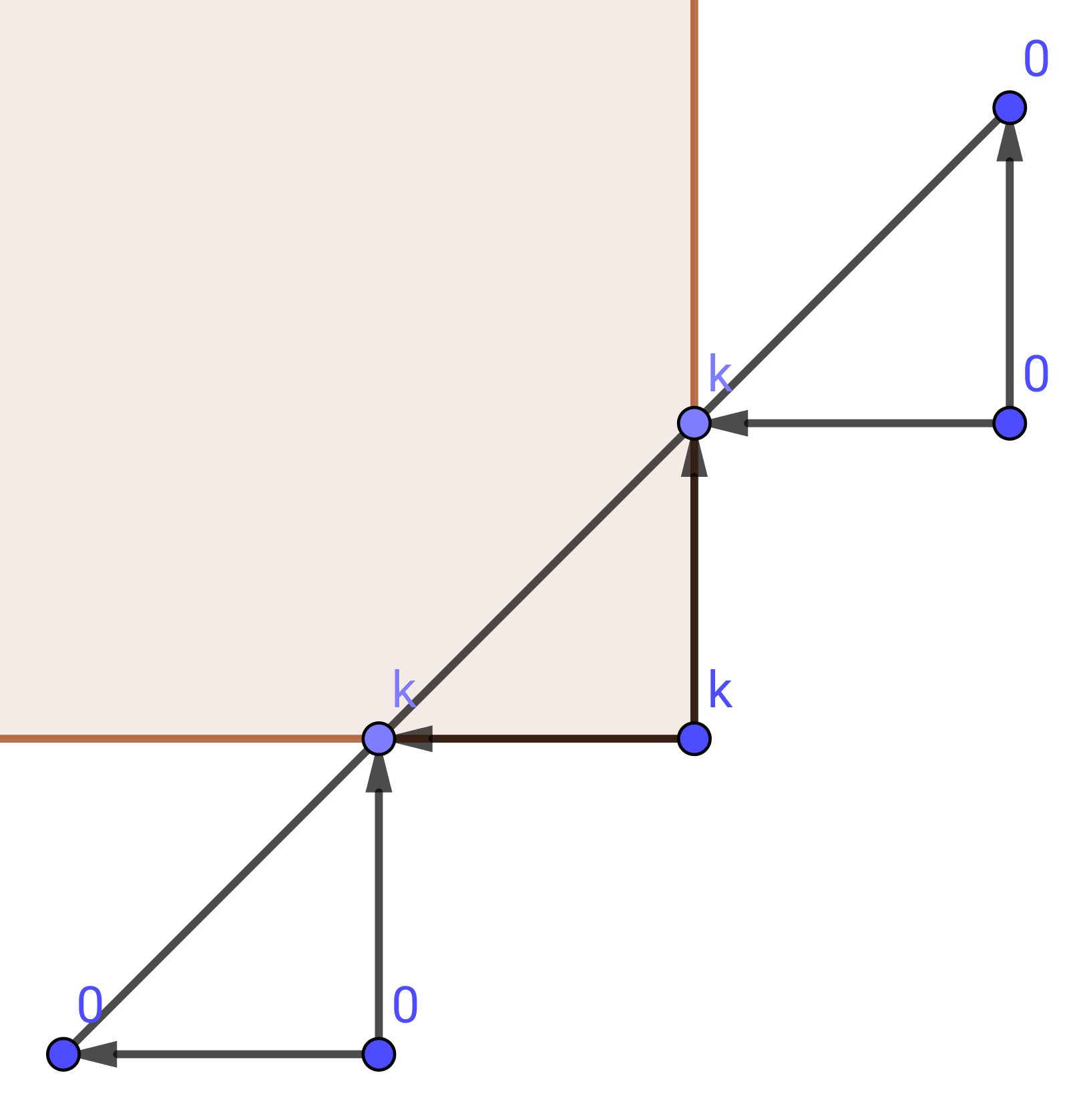}
	\caption{Extension of zigzag intervals to block intervals for the four different types $(\cdot,\cdot)$, $[\cdot,\cdot)$, $(\cdot,\cdot]$ and $[\cdot,\cdot]$ (in that order). Cf. Figure 3 in \cite{botnan2018}.}\label{fig:blocks}
\end{figure}
The following lemma motivates why $E$ is called the block extension functor. 
\begin{lemma}
	The block extension functor $E$ sends zigzag interval modules to block interval modules, i.e. it holds that $E(I_{\langle a,b\rangle_{ZZ}})=I_{\langle a,b\rangle_{BL}}$.
\end{lemma}
\begin{proof}
	This is a consequence of Lemma E.8 in \cite{kim2021arXivV4} (only in version four), which itself is a slight extension of Lemma 4.1 in \cite{botnan2018}. However, the reader can easily verify this by using the explicit formula in Lemma \ref{lemma_extension_functor_explicit}.
\end{proof}
In Figure \ref{fig:blocks}, one can see examples of zigzag interval modules and the corresponding blocks after applying the block extension functor $E$. The differences for the four different types are shown.\\
Under suitable finiteness assumptions the block extension functor preserves direct sums.
\begin{lemma}
	Let $M:ZZ\to\cat{vec}$ such that for all $\langle a,b\rangle_{ZZ}$, $\limit M|_{\langle a,b\rangle_{ZZ}}$ and $\colimit M|_{\langle a,b\rangle_{ZZ}}$ are finite dimensional. Then, if $M\cong\oplus_{k\in K}I_{\langle a_k,b_k\rangle_{ZZ}} $ then $E(M)\cong\oplus_{k\in K}I_{\langle a_k,b_k\rangle_{BL}}$.
\end{lemma}
\begin{proof}
	This is a consequence of Lemma E.9 in \cite{kim2021arXivV4} (only in version four).
\end{proof}
The next lemma shows how to actually calculate the components of $E(M)$ and is crucial for the proof of the stability of persistence landscapes. In Figure \ref{fig:extension_functor_explicit}, one can see how the components of $E(M)$ for a zigzag module $M$ are calculated.
\begin{figure}
	\centering
	\begin{subfigure}{0.5\textwidth}
		\centering
		\includegraphics[width=\textwidth]{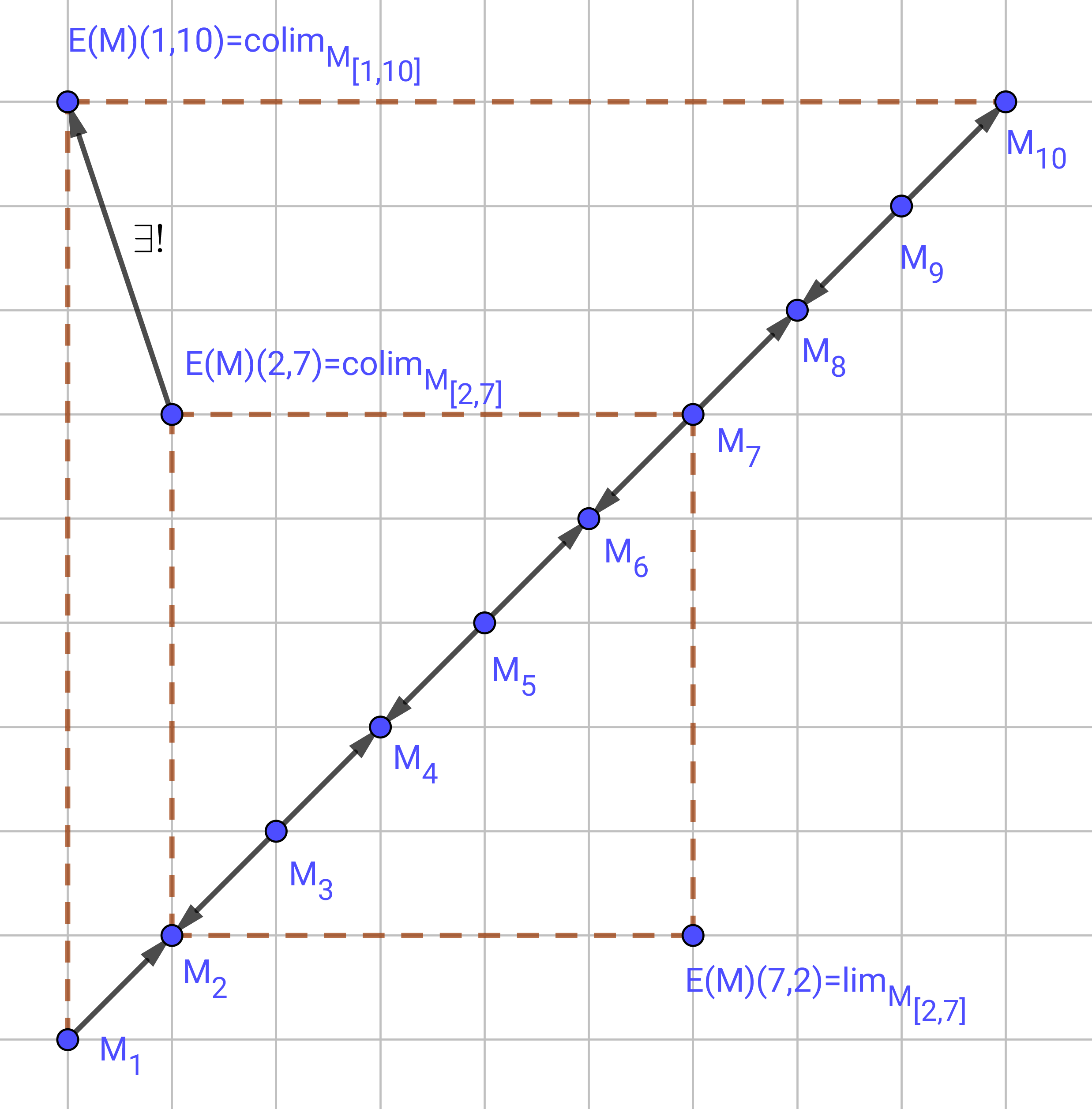}
		\subcaption{}
	\end{subfigure}\quad
	\begin{subfigure}{0.4\textwidth}
		\centering
		\includegraphics[width=0.8\textwidth]{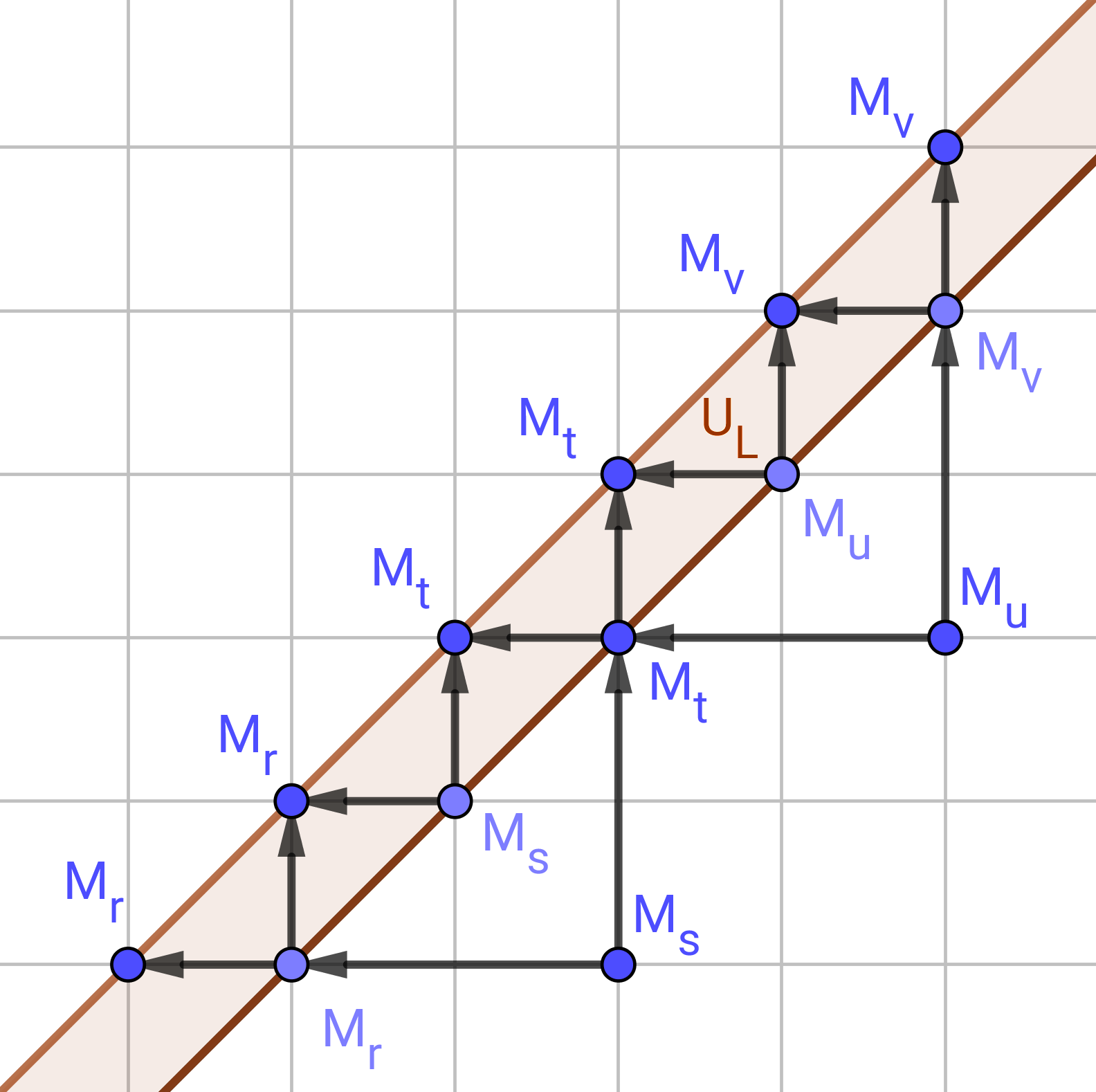}
		\subcaption{}\label{subfig:U_L}
	\end{subfigure}
	
	\caption{(a) A visualization of the block extension functor and Lemma \ref{lemma_extension_functor_explicit}. The unique map is given by the universal property of colimits. Cf. Figure 2 in \cite{botnan2018}. \\(b) The set $U_L$. }\label{fig:extension_functor_explicit}
\end{figure}
\begin{lemma}\label{lemma_extension_functor_explicit}
	For $(a,b)\in\Z^\op\times\Z$, it holds that
	\begin{align*}
		E(M)(a,b) = \begin{cases}
			\colimit M|_{[a,b]} & \mbox{ for } a\leq b,\\
			\limit M|_{[b,a]} & \mbox{ for } a> b.
		\end{cases}
	\end{align*}
	Furthermore, the structure maps of $E(M)$ are the maps given by the universal properties of limits and colimits, respectively.
\end{lemma}
\begin{proof}
	We first proof the part $E(M)(a,b)=\colimit M|_{[a,b]}$ for $a\leq b$. By definition of left Kan extensions, 
	\begin{align*}
		E(M)(a,b)=\colimit M|_{\{i\in\Z|\iota(i)\leq (a,b)\in\Z^\op\times\Z\}}.
	\end{align*}
	By definition of the inclusion $\iota:ZZ\hookrightarrow \Z^\op \times\Z$, we have
	\begin{align*}
		\{i\in\Z|\iota(i)\leq (a,b)\in\Z^\op\times\Z\} =\begin{cases}
			[a,b] & \mbox{if both $a$ and $b$ are source indices},\\
			[a,b+1] & \mbox{if $a$ is a source and $b$ is a sink index},\\
			[a-1,b+1] & \mbox{if both $a$ and $b$ are sink indices},\\
			[a-1,b] & \mbox{if $a$ is a sink and $b$ is a source index}.
		\end{cases}
	\end{align*}
	In each of this four cases $[a,b]$ is an upper fence of the respective interval. Using Proposition \ref{prop_iso_limcolim} we obtain 
	\begin{align*}
		E(M)(a,b)=\colimit M|_{\{i\in\Z|\iota(i)\leq (a,b)\in\Z^\op\times\Z\}}=\colimit M|_{[a,b]}.
	\end{align*}
	It remains to show that $E(M)(a,b)=\limit M|_{[b,a]}$ for $a>b$. We again use the observation in Proposition \ref{prop_iso_limcolim} that limits only depend on lower fences of posets. One lower fence of the poset $U$ is given by the set $U_L:=\{(i,j)|i\in\Z,j=i \mbox{ or }j=i+1\}$. We have that $E(M)(i,i)=M_i$ and
	\begin{align*}
		E(M)(i,i+1)=\begin{cases}
			M_i & \mbox{if $i$ is a sink index},\\
			M_{i+1} & \mbox{if $i$ is a source index}.
		\end{cases}
	\end{align*}
	See also Subfigure \ref{subfig:U_L}. By definition, we have
	\begin{align*}
		E(M)(a,b)=\limit (E_1(M)|_U)|_{S},
	\end{align*}
	with $S:=\{(i,j)\in U|(i,j)\geq (a,b)\}$. Observe that $S\cap U_L$ is a lower fence of $S\cap U$ and so 
	\begin{align*}
		E(M)(a,b)=\limit (E_1(M)|_U)|_{S}=\limit E_1(M)|_{S\cap U_L}.
	\end{align*}
	We calculate that 
	\begin{align*}
		S\cap U_L =\begin{cases}
			[b-1,a+1] & \mbox{if both $a$ and $b$ are source indices},\\
			[b,a+1] & \mbox{if $a$ is a source and $b$ is a sink index},\\
			[b,a] & \mbox{if both $a$ and $b$ are sink indices},\\
			[b-1,a] & \mbox{if $a$ is a sink and $b$ is a source index}.
		\end{cases}
	\end{align*}
	In each of these cases the interval $[b,a]$ is a lower fence and thus,
	\begin{align*}
		E(M)(a,b)=\limit M|_{[b,a]}.
	\end{align*}
\end{proof}

\subsection{Interleaving distance for extended zigzag modules}
We consider an extended zigzag module as a functor $M:ZZ\times \Z\to\cat{vec}$. In the following, by $M(\cdot,z)$ we mean the zigzag module $M|_{\{(x,z):x\in ZZ\}}$ and by $\mathcal{E}(M)(\cdot,\cdot,z)$ we mean the $\Z^\op\times \Z$-indexed module $\mathcal{E}(M)|_{\{(x,y,z):x\in\Z^\op,y\in\Z\}}$. We define a functor $\mathcal{E}:\cat{vec}^{ZZ\times\Z}\to\cat{vec}^{\Z^\op\times\Z\times\Z}$ as follows
\begin{enumerate}
	\item On objects: for an extended zigzag module $M\in\cat{vec}^{ZZ\times\Z}$ we define $\mathcal{E}(M)(\cdot,\cdot,z):=E((M)(\cdot,z))$. The vertical structure maps in $M$ can be seen as morphisms between the zigzag modules $M(\cdot,z_1)$ and $M(\cdot,z_2)$ for $z_1\leq z_2$. Hence, they induce morphisms $\mathcal{E}(M)(\cdot,\cdot,z_1)\to\mathcal{E}(M)(\cdot,\cdot,z_2)$, which themselves become the structure maps of $\mathcal{E}(M)$.
	\item On morphisms: given a morphism between extended zigzag modules $f:M\to N$, for each $z\in\Z$ we implicitly have a morphism between zigzag modules $f_z:M(\cdot,z)\to N(\cdot,z)$. Again by universality of limits and colimits we obtain a morphism $E(f_z):E(M(\cdot,z))\to E(N(\cdot,z))$ for each $z$. It remains to show that this is indeed a natural transformation, i.e. that the square commutes for all $z_1\leq z_2$
	\[	\begin{tikzcd}
		E(M(\cdot,z_1)) & E(M(\cdot,z_2)) \\
		E(N(\cdot,z_1)) & E(N(\cdot,z_2)),
		\arrow[from=1-1, to=1-2]
		\arrow["E(f_{z_1})",from=1-1, to=2-1]
		\arrow[from=2-1, to=2-2]
		\arrow["E(f_{z_2})",from=1-2, to=2-2]
	\end{tikzcd}
	\] 
	which again is a consequence of the universality of limits and colimits.
\end{enumerate}
Alternatively, one could define the functor $\mathcal{E}$ in analogy to the block extension functor $E$. For this, consider the inclusion $\tilde{\iota}:ZZ\times \mathbb{Z}\hookrightarrow \Z^\op\times\Z\times \Z$ defined as 
\[\tilde{\iota}(i,z)=\begin{cases}
	(i,i,z) & \mbox{if $i$ is a sink index,} \\
	(i+1,i-1,z)& \mbox{if $i$ is a source index}.
\end{cases}
\]
In this case, by source index we mean that $i$ is a source index in $ZZ$ and by sink index that $i$ is a sink index in $ZZ$. Now, regard the set \[\tilde{U}:=\{(x,y,z)\in\Z^\op\times\Z\times\Z|x\leq y\},\]
as well as the inclusion $\tilde{\kappa}:\tilde{U}\hookrightarrow \Z^\op\times\Z\times\Z$. Then, $\mathcal{E}$ can be defined as the composition of three functors
\begin{align*}
	\mathcal{E}:= \mathrm{Ran}_{\tilde{\kappa}}\circ (-)|_{\tilde{U}}\circ\mathrm{Lan}_{\tilde{\iota}}:\cat{vec}^{ZZ\times\Z}\to\cat{vec}^{\Z^\op\times\Z\times\Z}.
\end{align*}
To see that $\mathcal{E}(M)(\cdot,\cdot,z)=E(M(\cdot,z))$ it is again sufficient to recall the considerations in Lemma \ref{lemma_extension_functor_explicit} and the fact that limits and colimits depend only on lower and upper fences, respectively. Indeed, we get the explicit formulation
\begin{align*}
	\mathcal{E}(M)(x,y,z) := \begin{cases}
		\colimit M(\cdot,z)|_{[x,y]} & \mbox{ for } x\leq y,\\
		\limit M(\cdot,z)|_{[y,x]} & \mbox{ for } x> y.
	\end{cases}
\end{align*}
The structure maps are, again, maps that are given by the universal properties of limits and colimits.\\
The subsequent proposition is an extension of Prop. 4.3 in \cite{botnan2018} and it will not be used anywhere else in this paper. We state it for sake of completeness.
\begin{proposition}
	The functor $\mathcal{E}$ is fully faithful, i.e. the set of morphisms from $M$ to $N$ is isomorphic to the set of morphisms from $\mathcal{E}(M)$ to $\mathcal{E}(N)$ for all extended zigzag modules $M$ and $N$.
\end{proposition}
\begin{proof}
	The first part of the proof that $(-)|_{\tilde{U}}\circ \mathrm{Lan}_{\tilde{\iota}}$ is fully faithful is analogue to the proof of Prop. 4.3 in \cite{botnan2018}. We formulate the dual arguments to show that $\mathrm{Ran}_{\tilde{\kappa}}$ is fully faithful. It is known that the right Kan extension is right adjoint to the restriction functor $(-)|_{\tilde{U}}$ (see \cite{riehl_categorical}, (1.1)). By Theorem IV.3.1 in \cite{maclane_categories}, a right adjoint is fully faithful if and only if the counit of the adjunction is a natural isomorphism. This is easy to see since $(-)|_{\tilde{U}}\circ \mathrm{Ran}_{\tilde{\kappa}}(M)\cong M$ for any extended zigzag module $M$. In total, $\mathcal{E}$ is fully faithful as a composition of fully faithful functors.
\end{proof}
The following definitions are standard in the theory of multiparameter persistent homology and are adapted to our setting. Let $v=(v_1,v_2,v_3)\in\Z^\op\times\Z\times\Z$ be such that $v_1\leq 0$ and $v_2,v_3\geq 0$.\\
\begin{definition}(Shift functors) \\
	The $v$-shift functor $(-)_v:\cat{vec}^{\Z^\op\times\Z\times\Z}\to\cat{vec}^{\Z^\op\times\Z\times\Z}$ is defined as follows:
	\begin{enumerate}
		\item For $F\in\cat{vec}^{\Z^\op\times\Z\times\Z}$, we define $F_v\in\cat{vec}^{\Z^\op\times\Z\times\Z}$ in the following way: $F_v(x)=F(x+v)$ for all $x\in\Z^\op\times\Z\times\Z$ and $F_v(x\leq x'):=F(x+v\leq x'+v)$ for all $x\leq x'\in \Z^\op\times\Z\times\Z$ and all structure maps.
		\item Let $F,G\in \cat{vec}^{\Z^\op\times\Z\times\Z}$. For any morphism $\eta:F\to G$, the corresponding morphism $\eta_v:F_v\to G_v$ is defined as $\eta_v(x)=\eta(x+v):F_v(x)\to G_v(x)$ for all $x\in\Z^\op\times\Z\times\Z$.
	\end{enumerate}
\end{definition}
\begin{definition}($v$-interleaving)\\
	We say that $F,G\in\cat{vec}^{\Z^\op\times\Z\times\Z}$ are $v$-interleaved if there are natural transformations $\eta:F\to G_v$ and $\nu:G\to F_v$ such that 
	\begin{enumerate}
		\item $\nu_v \circ\eta = \psi_F^{2v}$,
		\item $\eta_v\circ \nu= \psi_G^{2v}$,
	\end{enumerate}
	where $\psi_F^{v}:F\to F_v$ denotes the natural transformation whose restriction to $F(x)$ is the linear map $F(x\leq x+v)$.
\end{definition}
Finally, we can define the interleaving distance for extended zigzag modules.
\begin{definition}(Interleaving distance)\\
	Let $\vec{\epsilon}=\epsilon(-1,1,1)$ with $\epsilon\geq 0$ and let further $F,G\in\cat{vec}^{\Z^\op\times\Z\times\Z}$. The interleaving distance between $F$ and $G$ is defined as
	\begin{align*}
		d_I(F,G):=\inf \{\epsilon\geq 0:\; F,G \mbox{ are $\vec{\epsilon}$-interleaved}\}
	\end{align*}
	and as $d_I(F,G)=\infty$ if there is no $\vec{\epsilon}$-interleaving. \\
	For extended zigzag modules $M,N$ we define the interleaving distance as
	\begin{align*}
		d_I(M,N):=d_I(\mathcal{E}(M),\mathcal{E}(N)).
	\end{align*} 
\end{definition}
Having the necessary definitions, we state the main result of this section.
\begin{theorem}(Stability of spatiotemporal persistence landscapes)\\
	Let $M,N$ be extended zigzag modules. It holds
	\begin{align*}
		d_\lambda^\infty(M,N)\leq d_I(M,N).
	\end{align*}
\end{theorem}
\begin{proof}
	Assume that $M,N$ are $\epsilon$-interleaved. Let $(x,z)\in ZZ\times \Z$. Without loss of generality let $r=\lambda_k(M)(x,z)\geq \lambda_k(N)(x,z)$ and let $\lambda_k(M)(x,z)\geq \epsilon$. Since $M,N$ are $\epsilon$-interleaved we obtain the commutative diagram
	\[
	\begin{tikzcd}
		\mathcal{E}(M)(x+r,x-r,z-r) & \mathcal{E}(M)(x-r,x+r,z+r) \\
		\mathcal{E}(N)(x+h,x-h,z-h) & \mathcal{E}(N)(x-h,x+h,z+h)
		\arrow[from=1-1, to=1-2]
		\arrow[from=1-1, to=2-1]
		\arrow[from=2-1, to=2-2]
		\arrow[from=2-2, to=1-2]
	\end{tikzcd}
	\]
	where $h=r-\epsilon$. Using Lemma \ref{lemma_extension_functor_explicit}, we obtain the commutative diagram
	\[\begin{tikzcd}
		\limit M(\cdot,z-r)|_{[x-r,x+r]} & \colimit M(\cdot,z+r)|_{[x-r,x+r]} \\
		\limit N(\cdot,z-h)|_{[x-h,x+h]} & \colimit N(\cdot,z+h)|_{[x-h,x+h]}
		\arrow["\phi",from=1-1, to=1-2]
		\arrow[from=1-1, to=2-1]
		\arrow["\psi",from=2-1, to=2-2]
		\arrow[from=2-2, to=1-2]
	\end{tikzcd}.\]
	It follows that $k=\rank\; M|_{R_{(x,z)}^r}=\rank\; \phi\leq \rank\;\psi=\rank \;N|_{R_{(x,z)}^{r-\epsilon}}$. Hence, $\lambda_k(N)(x,z)\geq r-\epsilon$ and so $\lambda_k(M)(x,z)-\lambda_k(N)(x,z)\leq \epsilon$, completing the proof.
\end{proof}
\begin{remark}
	In the case where one would like to consider rectangular regions as mentioned in Remark \ref{rem:rectangular_regions}, an analogous stability theorem can be obtained by considering an adapted inclusion from $ZZ\times \Z$ to $\Z^\op\times\Z\times\Z$ and therefore also an adapted interleaving distance.
\end{remark}

\section{Algorithm}\label{sec:algorithm}
In practice, we regard finite and discrete time series and restrict our calculations to finitely many values of the distance parameter $\epsilon$. Therefore, the obtained persistence module has the following form:
\begin{equation*}
	\begin{tikzcd}[ampersand replacement=\&]
		M_{1,1} \arrow[r]\arrow[d]\&  M_{1,2}\arrow[d] \& \arrow[l] M_{1,3}\arrow[d]\arrow[r] \& M_{1,4}\arrow[d] \& \arrow[l]\cdots\arrow[r] \& M_{1,m-1}\arrow[d] \&  \arrow[l]M_{1,m}\arrow[d] \\
		M_{2,1} \arrow[r]\arrow[d]\&  M_{2,2}\arrow[d] \& \arrow[l] M_{2,3}\arrow[d]\arrow[r] \& M_{2,4}\arrow[d] \& \arrow[l]\cdots\arrow[r] \& M_{2,m-1}\arrow[d] \&  \arrow[l]M_{2,m}\arrow[d] \\
		\vdots \arrow[d]\&  \vdots\arrow[d] \&  \vdots\arrow[d] \& \vdots\arrow[d] \& \ddots\& \vdots\arrow[d] \&  \vdots\arrow[d] \\
		M_{n,1} \arrow[r]\&  M_{n,2} \& \arrow[l] M_{n,3}\arrow[r] \& M_{n,4} \& \arrow[l]\cdots\arrow[r] \& M_{n,m-1} \&  \arrow[l]M_{n,m}
	\end{tikzcd}
\end{equation*}

Regarding the definition of persistence landscapes (Def. \ref{def:landsc}), we are interested in the rank of $M$ restricted to quadratic regions $R_x^\epsilon$ in the parameter space centered at a point $x$. According to Theorem \ref{theo_main}, the generalized rank of an interval in the persistence module $M$ can be computed as the rank of the module restricted to a zigzag path along certain boundary points of $M$. To be precise, this path starts with a path through the lower fence and ends with a path through the upper fence. It holds that the lower fence contains all minimal elements and the upper fence contains all maximal elements.
In the case of squares $R_x^\epsilon$, the minimal and maximal points are points on the lower and upper edge of the square. For example, in the following diagram the minimal elements are colored in blue and the maximal elements are colored in green. The respective lower and upper fence is denoted by colored arrows in red and orange, respectively.
\begin{equation*}
	\begin{tikzcd}[ampersand replacement=\&]
		\textcolor{blue}{M_{1,1}} \arrow[r,red]\arrow[d]\&  M_{1,2}\arrow[d] \& \arrow[l,red] \textcolor{blue}{M_{1,3}}\arrow[d]\arrow[r,red] \& M_{1,4}\arrow[d] \& \arrow[l,red]\textcolor{blue}{M_{1,5}}\arrow[d] \\
		M_{2,1} \arrow[r]\arrow[d]\&  M_{2,2}\arrow[d] \& \arrow[l] M_{2,3}\arrow[d]\arrow[r] \& M_{2,4}\arrow[d] \& \arrow[l] M_{2,5}\arrow[d] \\
		M_{3,1}\arrow[r] \arrow[d]\&  M_{3,2}\arrow[d] \& \arrow[l] M_{3,3}\arrow[d]\arrow[r] \& M_{3,4}\arrow[d] \&\arrow[l] M_{3,5} \arrow[d] \\
		M_{4,1}\arrow[r] \arrow[d]\&  M_{4,2}\arrow[d] \& \arrow[l] M_{4,3}\arrow[d]\arrow[r] \& M_{4,4}\arrow[d] \&\arrow[l] M_{4,5} \arrow[d] \\
		M_{5,1} \arrow[r,orange]\&  \textcolor{green}{M_{5,2}} \& \arrow[l,orange] M_{5,3}\arrow[r,orange] \& \textcolor{green}{M_{5,4}} \& \arrow[l,orange] M_{5,5} 	
	\end{tikzcd}
\end{equation*}
We connect the lower and upper fence by the path that goes through the next smaller square centered at the same point to implicitly compute the rank of the next smaller square. We repeat this procedure iteratively. This is based on the assumptions of Theorem \ref{theo_main}, that state that the only condition for the intermediate path is that it has to connect the lower and the upper fence. As a result, for every point $x$ we obtain one zigzag diagram along the path that contains the information of the rank of every square centered at point $x$. Such a path at point $x=(4,4)$ would look like the red arrows in the following diagram.
\begin{equation*}
	\begin{tikzcd}[ampersand replacement=\&]
		\textcolor{blue}{M_{1,1}} \arrow[r,red, thick]\arrow[d]\&  M_{1,2}\arrow[d] \& \arrow[l,red, thick] \textcolor{blue}{M_{1,3}}\arrow[d]\arrow[r,red, thick] \& M_{1,4}\arrow[d] \& \arrow[l,red, thick]\textcolor{blue}{M_{1,5}}\arrow[d] \arrow[r,red, thick]\& M_{1,6}\arrow[d] \& \arrow[l,red, thick]\textcolor{blue}{M_{1,7}}\arrow[d,red, thick] \\
		M_{2,1} \arrow[r]\arrow[d]\&  M_{2,2}\arrow[d,red, thick] \& \arrow[l,red, thick] \textcolor{blue}{M_{2,3}}\arrow[d]\arrow[r,red, thick] \& M_{2,4}\arrow[d] \& \arrow[l,red, thick] \textcolor{blue}{M_{2,5}}\arrow[d]\arrow[r,red, thick]\& M_{2,6}\arrow[d] \& \arrow[l,red, thick]M_{2,7}\arrow[d] \\
		M_{3,1}\arrow[r] \arrow[d]\&  M_{3,2}\arrow[d] \& \arrow[l,red, thick] \textcolor{blue}{M_{3,3}}\arrow[d]\arrow[r,red, thick] \& M_{3,4}\arrow[d] \&\arrow[l,red, thick] \textcolor{blue}{M_{3,5}} \arrow[d,red, thick] \arrow[r]\& M_{3,6}\arrow[d] \& \arrow[l]M_{3,7}\arrow[d]\\
		M_{4,1}\arrow[r] \arrow[d]\&  M_{4,2}\arrow[d] \& \arrow[l] M_{4,3}\arrow[d,red, thick]\arrow[r,red, thick] \& M_{4,4}\arrow[d] \&\arrow[l,red, thick] M_{4,5} \arrow[d]\arrow[r]\& M_{4,6}\arrow[d] \& \arrow[l]M_{4,7}\arrow[d] \\
		M_{5,1}\arrow[r] \arrow[d]\&  M_{5,2}\arrow[d] \& \arrow[l] M_{5,3}\arrow[d]\arrow[r,red, thick] \& \textcolor{green}{M_{5,4}}\arrow[d] \&\arrow[l,red, thick] M_{5,5} \arrow[d] \arrow[r,red, thick]\& M_{5,6}\arrow[d,red, thick] \& \arrow[l]M_{5,7}\arrow[d]\\
		M_{6,1}\arrow[r,red, thick] \arrow[d,red, thick]\&  \textcolor{green}{M_{6,2}}\arrow[d] \& \arrow[l,red, thick] M_{6,3}\arrow[d]\arrow[r,red, thick] \& \textcolor{green}{M_{6,4}}\arrow[d] \&\arrow[l,red, thick] M_{6,5} \arrow[d]\arrow[r,red, thick]\& \textcolor{green}{M_{6,6}}\arrow[d] \& \arrow[l]M_{6,7}\arrow[d] \\
		M_{7,1} \arrow[r,red, thick]\&  \textcolor{green}{M_{7,2}} \& \arrow[l,red, thick] M_{7,3}\arrow[r,red, thick] \& \textcolor{green}{M_{7,4}} \& \arrow[l,red, thick] M_{7,5} \arrow[r,red, thick]\& \textcolor{green}{M_{7,6}} \& \arrow[l,red, thick]M_{7,7}
	\end{tikzcd}
\end{equation*}
To compute all spatiotemporal persistence landscapes at all points $x$ in the parameter space of size $n\times m$, one has to calculate $n\cdot m$ barcodes. Since they are independent of each other this step can be parallelized. In view of Lemma \ref{lemma:fibered_barcode} for the multiparameter persistence landscapes one has to compute only $n+m-1$ barcodes.\\
For every zigzag path in the parameter space, we calculate the sequence of simplicial complexes along this path. Then, we calculate the barcodes along this paths using \textsc{FastZigzag} \cite{dey_fzz}, an algorithm to obtain a barcode from a zigzag filtration in $\mathcal{O}(k^\omega)$ time, where $k$ is the length of the input filtration and $\omega$ is the matrix multiplication exponent. Since \textsc{FastZigzag} requires a simplexwise filtration as input we convert every zigzag sequence into a simplexwise zigzag sequence. 
Unfortunately, in our case the length of the input filtration can get very large since when going from one to the next window we delete all simplices of one window and add all simplices of the next window. At worst case, the number of possible insertions or deletions for one arrow in the zigzag module could be $2^N$ with $N$ being the number of points in the window. This is a very unlikely case, but in this case the length of the zigzag filtration is quadratic in $\min(2t-1,n)$, where $n$ is the number of spatial parameter values $\epsilon_1,...,\epsilon_n$ and $t$ number of windows. Furthermore, the length of the zigzag filtration is exponential in the number of points per window. From this we deduce that it is much more efficient to segment the time series such that we have as little points per windows as necessary. For periodic time series it means that one should aim to choose the window size as the length of the period.\\
In our implementation, all the simplicial complexes are Vietoris-Rips complexes. Since this is a costly type of simplicial complex, the choice of e.g. the alpha complex is expected to speed up calculations.\\ 
Given the barcodes along the respective paths, we can directly compute the landscape. Note that one has to convert back from the simplexwise barcode to the barcode of the original filtration. By considering Remark \ref{rem:kmax} one can easily calculate the landscapes from the barcodes. Our implementation uses parts of the code of \textsc{Topcat} \cite{topcat}, a library for multiparameter persistent homology available on GitHub.

\section{Applications}\label{sec:applications}
In this section, we use spatiotemporal persistence landscapes to visualize the persistent features in space and time of simulated time series.

\subsection{Sinusoidal signal}
First, we apply our method to two time series that consist of a noisy sine function with added white Gaussian noise, having a signal to noise ratio of 30 dB. Both can be seen in the first row of Figure \ref{fig:sinus_data}. The first one is an ordinary sinus whereas in the second time series, there is a jump at $x=0$. To obtain a point cloud, we perform a time-delay embedding with embedding dimension 2 (see Subsection \ref{subsubsec:delay_emb}). It is known that the resulting point cloud has the shape of an ellipse \cite{perea2015}.
\begin{figure}
	\centering
	\includegraphics[width=0.49\textwidth]{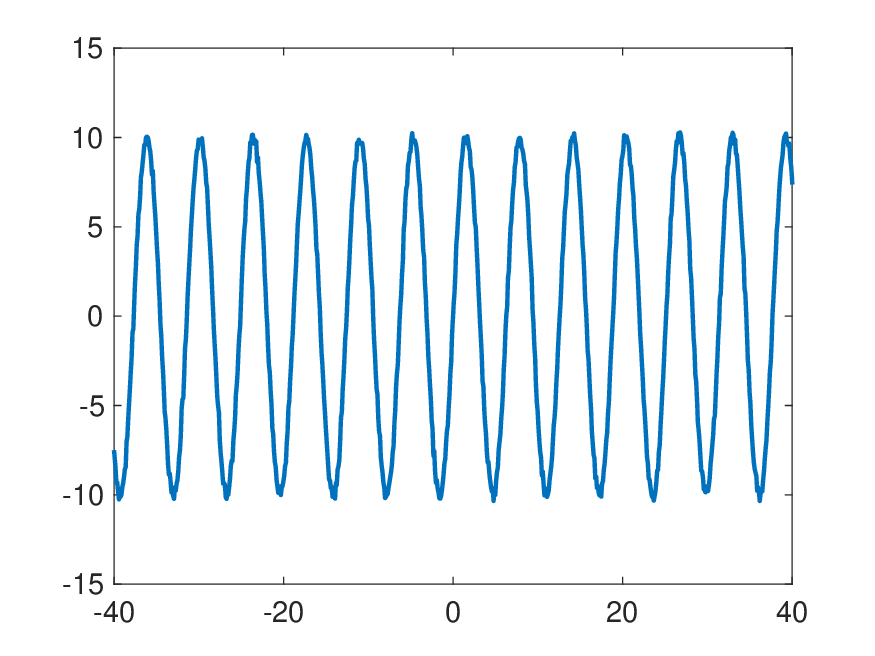}\;
	\includegraphics[width=0.49\textwidth]{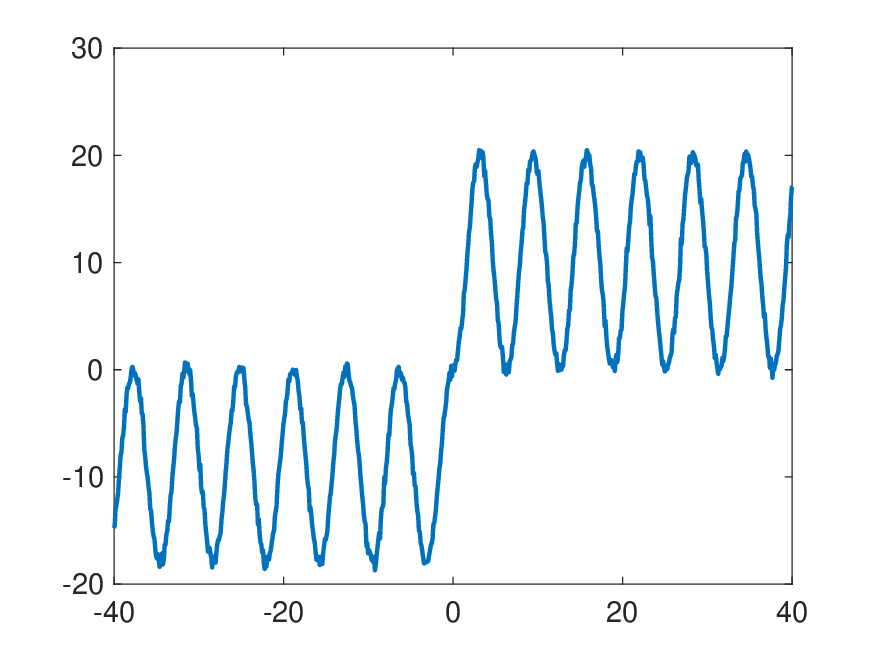}\\
	\includegraphics[width=0.49\textwidth]{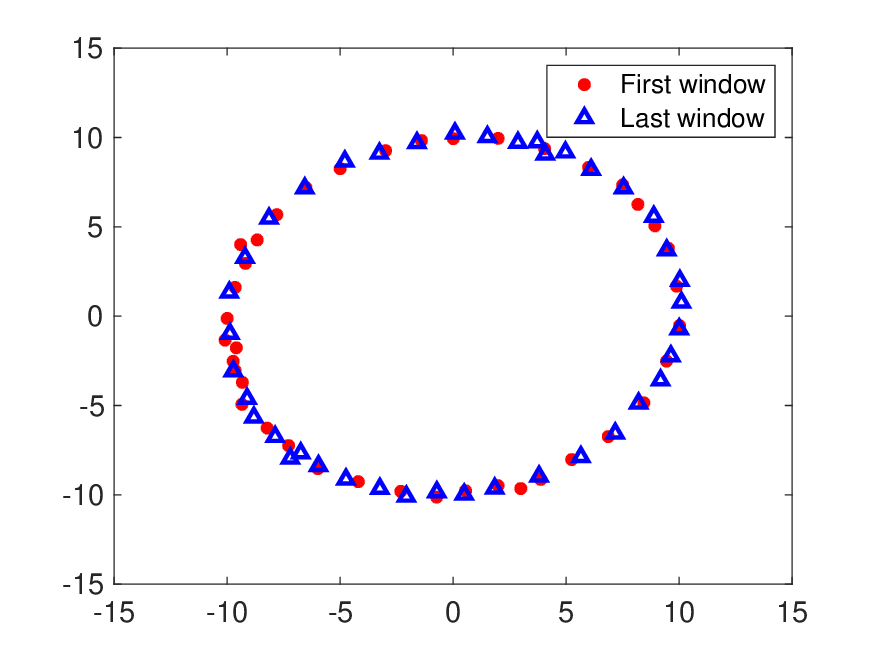}\;
	\includegraphics[width=0.49\textwidth]{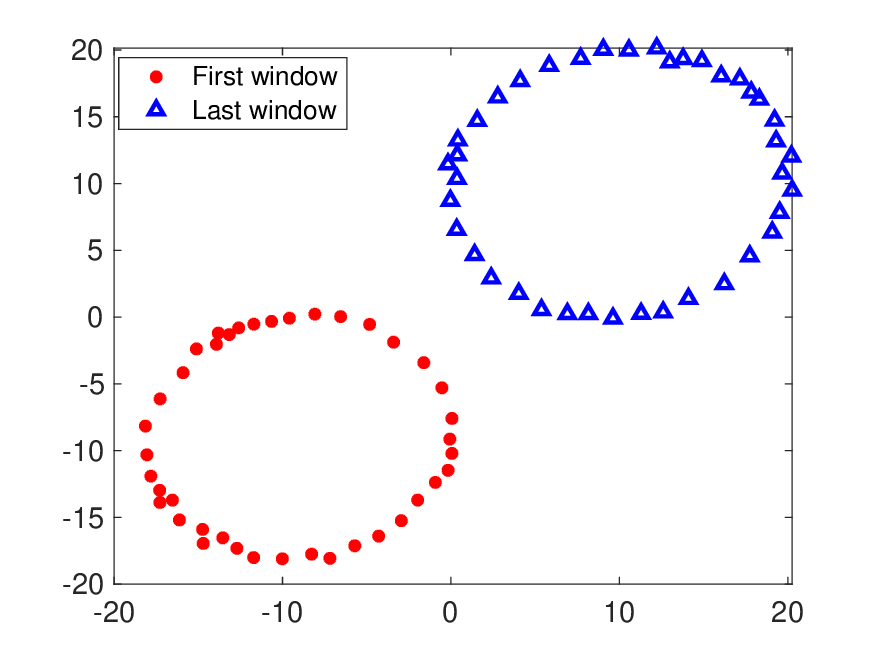}\\
	\includegraphics[width=0.49\textwidth]{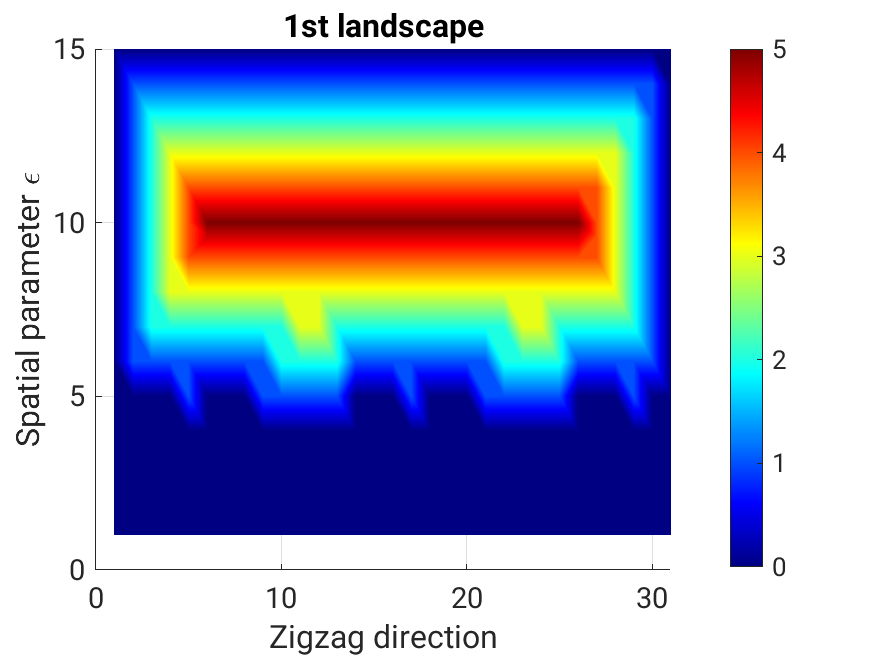}\;
	\includegraphics[width=0.49\textwidth]{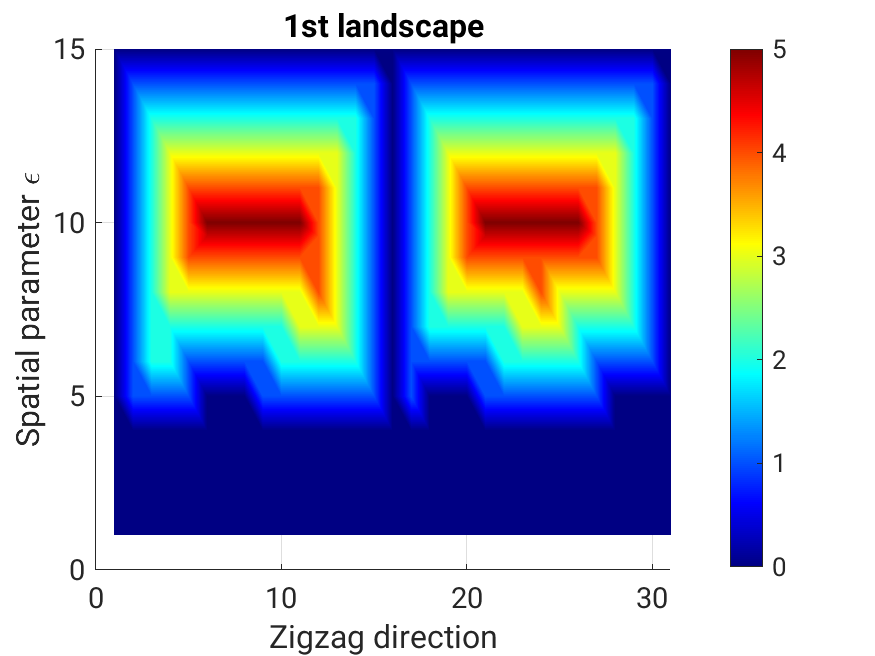}
	\caption{First row: noisy data; second row: delay embedding of the first and last windows; third row: first persistence landscape.}\label{fig:sinus_data}
\end{figure}
In the first row of Figure \ref{fig:sinus_data}, one can only see the central part of the used time series. For our calculations, we used longer time series lasting from -250s to 250s. We partition the time series into 16 windows and sample it down to 40 points per window by using a kmeans clustering algorithm. In the second column, one can see the point clouds of the first and the last windows of both time series, respectively. The jump that happens in the second time series at $t=0$ results in a shift of the ellipse obtained by delay embedding. To be precise, for the second time series the first 8 windows consist of points that approximate an ellipse in the third quadrant whereas the next 8 windows consist of points that approximate an ellipse in the first quadrant. Opposed to that, the location of the ellipse of each window obtained by the first time series does not change. In the second row of Figure \ref{fig:sinus_data} one can see the point clouds of the fist and last windows of both time series; on the left for the sinus time series and on the right for the sinus with the jump. Hence, for the first time series we expect one feature in homological dimension 1 that lasts from the first window to the last window. In contrast to that, in case of the second time series we expect two one dimensional features, one of them living in the first eight windows and the other one living in the second half of the windows.\\ 
In the third row of Figure \ref{fig:sinus_data}, one can see the resulting first persistence landscape for homological dimension 1 of both time series. Again, the landscape of the sinus time series is on the left and the landscape of the sinus with the jump is on the right. We only show the first landscape because in both cases already the second landscape and hence all higher landscapes are zero. We can conclude that there is only one persistent feature at each time. The landscapes for both functions look similar despite the fact that for the second function, there is a break at the vertical center line where the landscape decrease to zero. Hence, we see that there is no overlap between the two one dimensional homological features in the time series with the jump. Those homological features correspond to the ellipses in the first and second half of the time series. In contrast to that, for the first time series there is a homological feature that persists through the entire time series which causes the only elevation of the landscape. This qualitative difference between the two time series is captured well by our novel invariant.

\subsection{Detection of Hopf bifurcations}
We simulate an artificial data set following the procedure in \cite{tymochko2020} in order to detect Hopf bifurcations with spatiotemporal persistence landscapes. For this, we use the Sel'kov model \cite{selkov1968}, described by the following system of ordinary differential equations
\begin{align*}
	\dot{x}=-x+ay+x^2y, \quad \dot{y}=b-ay-x^2y.
\end{align*}
This system of equations has been observed to model the self-oscillations appearing in glycosis, the metabolic pathway to convert sugar into energy. By fixing the parameter $a=0.1$ and varying $b$ we obtain a limit cycle for $0.4\leq b\leq 0.8$ \cite{tymochko2020}. We model the observation function by taking only the $x$-coordinates as time series and use time delay embedding with embedding dimension $d=2$ (see Subsection \ref{subsubsec:delay_emb}) to obtain a point cloud in the plane. For each value of $b$, we cut the first half of the solution to get the limit behavior of the system and sample the point cloud down to 40 points by using kmeans clustering.

\begin{figure}
	\centering
	\includegraphics[width=0.32\textwidth]{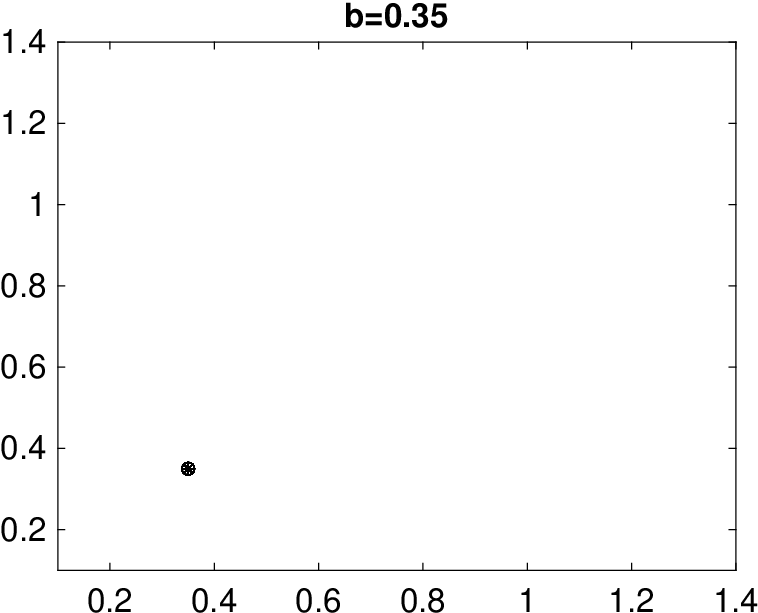}\;
	\includegraphics[width=0.32\textwidth]{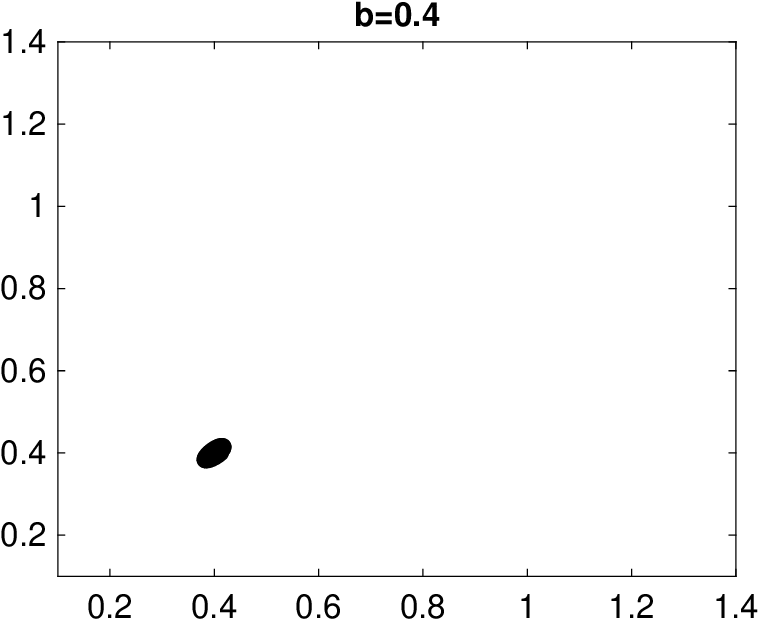}\;
	\includegraphics[width=0.32\textwidth]{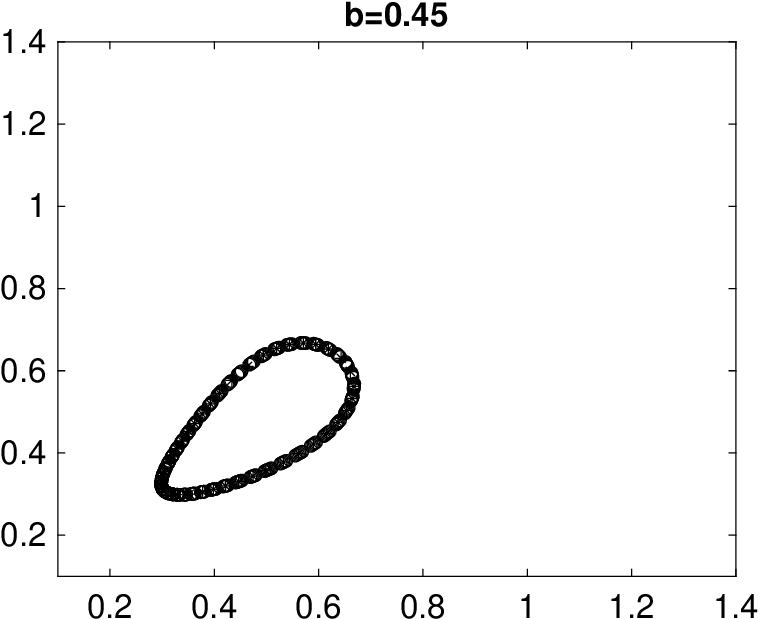}\\
	\vspace{0.2cm}
	\includegraphics[width=0.32\textwidth]{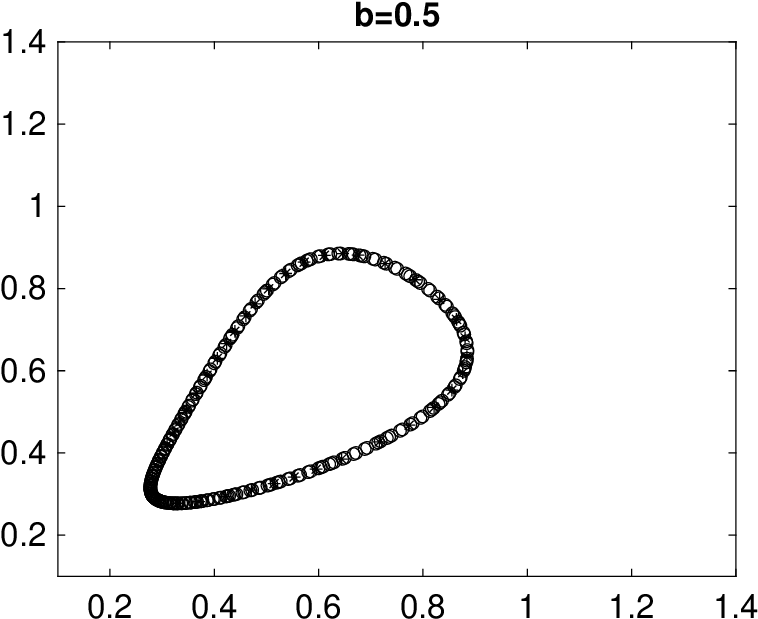}\;
	\includegraphics[width=0.32\textwidth]{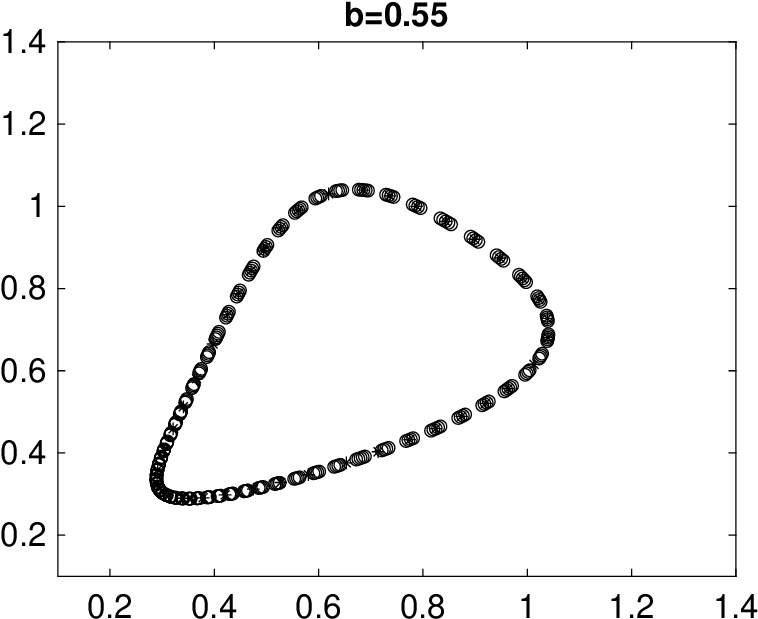}\;
	\includegraphics[width=0.32\textwidth]{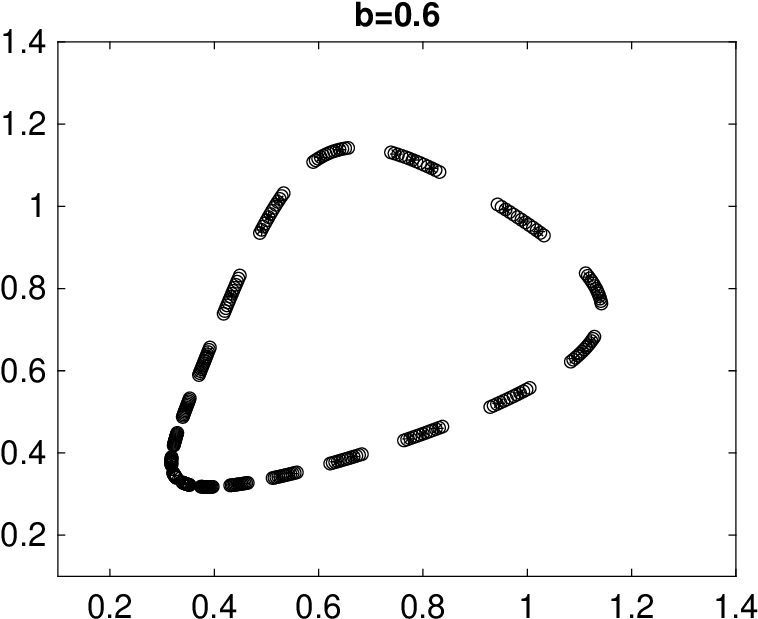}\\
	\vspace{0.2cm}
	\includegraphics[width=0.32\textwidth]{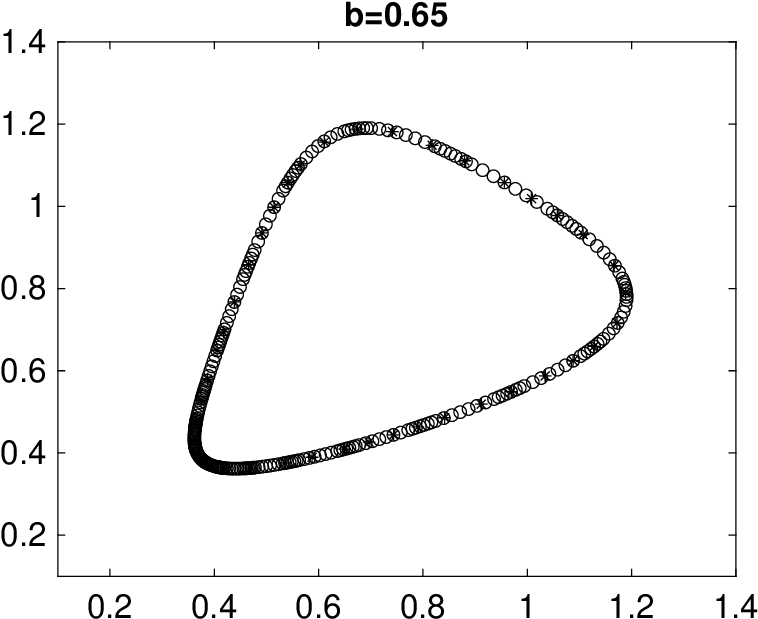}\;
	\includegraphics[width=0.32\textwidth]{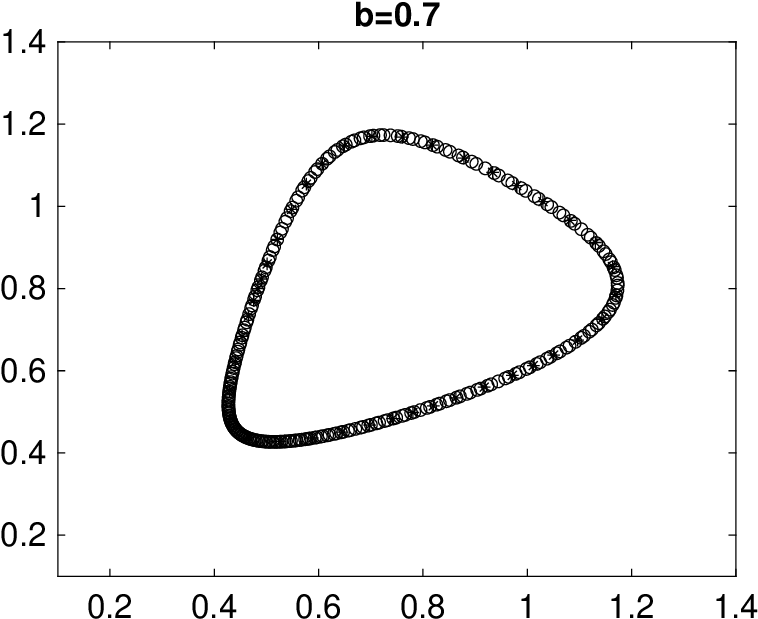}\;
	\includegraphics[width=0.32\textwidth]{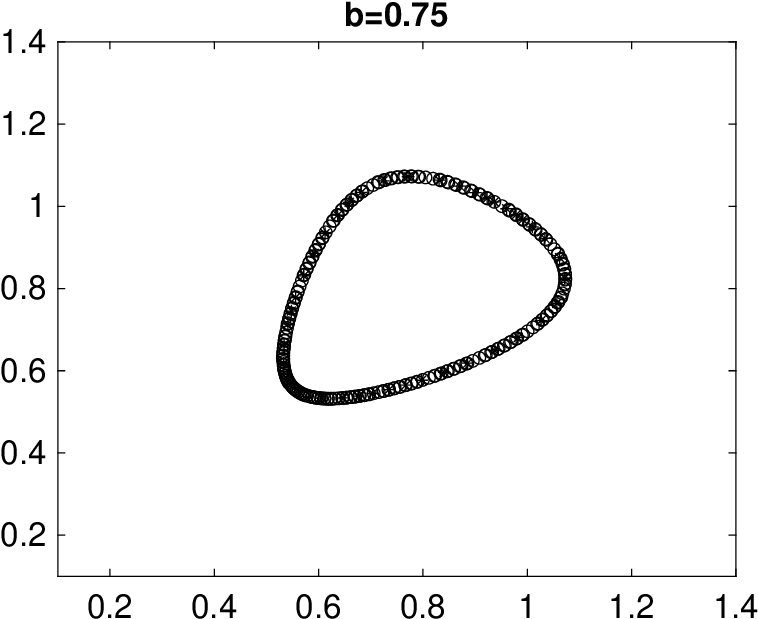}\\
	\vspace{0.2cm}
	\includegraphics[width=0.32\textwidth]{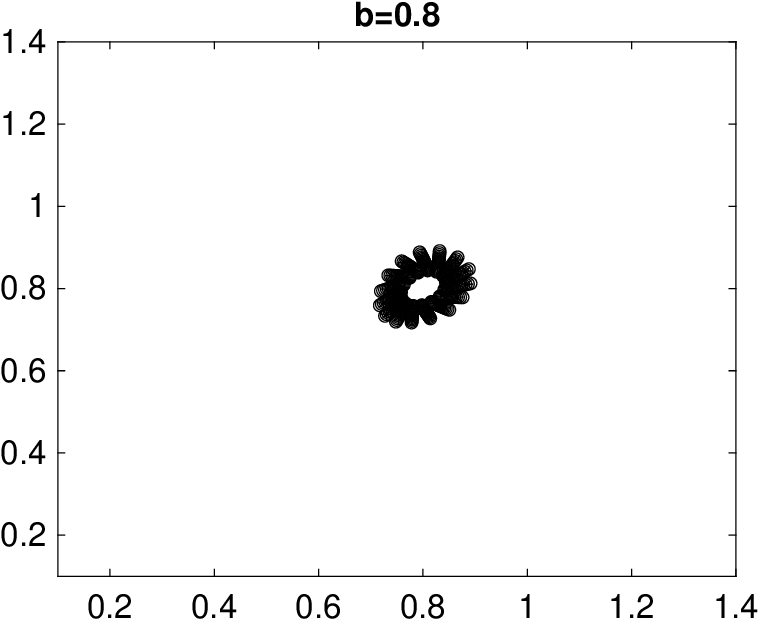}\;
	\includegraphics[width=0.32\textwidth]{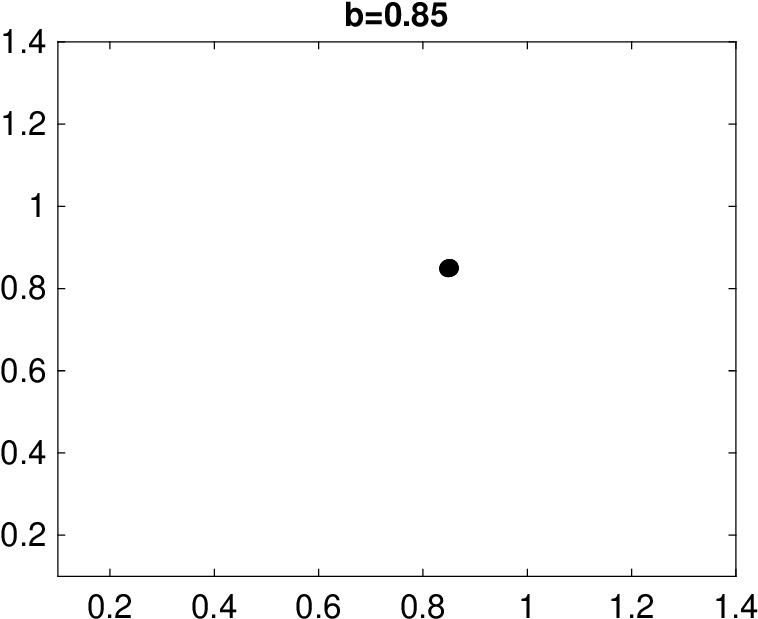}\;
	\includegraphics[width=0.32\textwidth]{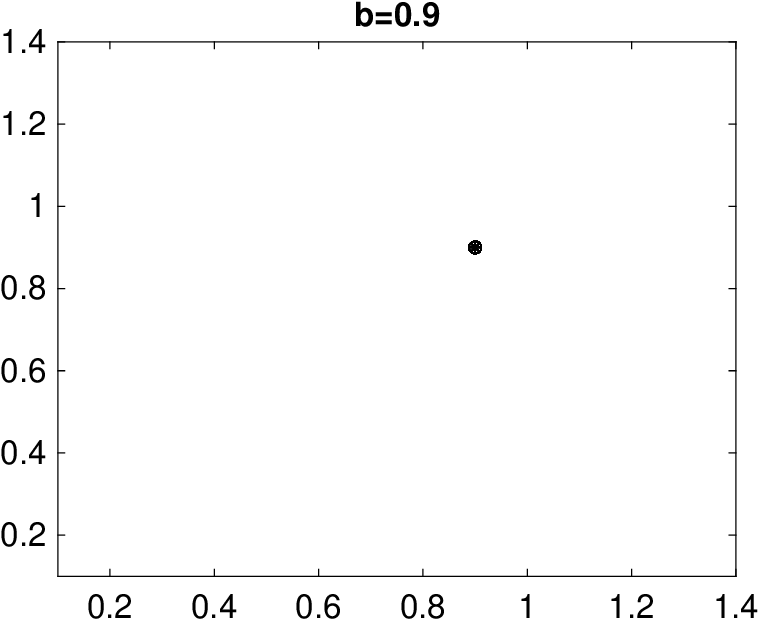}\\
	\vspace{0.2cm}
	\caption{Solutions of the Selkov model for different choices of parameter $b$.}\label{fig:selkov_data}
\end{figure}

In Figure \ref{fig:selkov_data}, one can see the solution for the Sel'kov model for parameter values ranging from $b=0.35$ to $b=0.9$ in steps of $0.05$. We simulated noisy data by adding white Gaussian noise to the data shown in Figure \ref{fig:selkov_data} with a SNR of $30$. For $30$ different noisy measurements, we calculated the first spatiotemporal persistent landscape in homological dimension one and the mean landscape. In Figure \ref{fig:lands_selkov}, one can see the resulting landscapes. In the top row on the left hand side, the landscape for the data set without noise is shown, which we refer to as the \textit{ground truth image}, on the right hand side one can see the mean landscape of the $30$ different noisy time series. The other six landscapes are the solutions for the first six noisy time series. We expect a homological feature in dimension one from $b=0.45$, which corresponds to zigzag index $5$. However, as we can see in the ground truth data, we cannot find any value for $\epsilon$ such that this feature survives on a region with radius more than one. The ground truth landscape has the highest value of $5$ for zigzag indices $12$ and $13$ and $\epsilon=10$. Hence, there is a persistent feature in space and time ranging from zigzag indices $7$ to $18$, meaning that we could detect a hole in the data set ranging from $b=0.5$ to $b=0.75$, which is approximately in accordance to the expectations. Even though the maximum value of the mean landscape is less that the maximum value of the ground truth, the landscape still indicated that there is a relevant hole ranging from zigzag indices $6$ to $17$. In the landscape of the first ($n=1$) time series, one does not see the qualitative behavior as in the other time series since there is no persistent feature in the expected range. This shows, how in some cases noise can affect the results such that the features loose persistence. On the contrary, the mean landscape still captures the relevant behavior of the data set and succeeds in detecting the Hopf bifurcations approximately, showing the relevance of statistical methods when dealing with noisy data.

\begin{figure}
	\centering
	\includegraphics[width=0.49\textwidth]{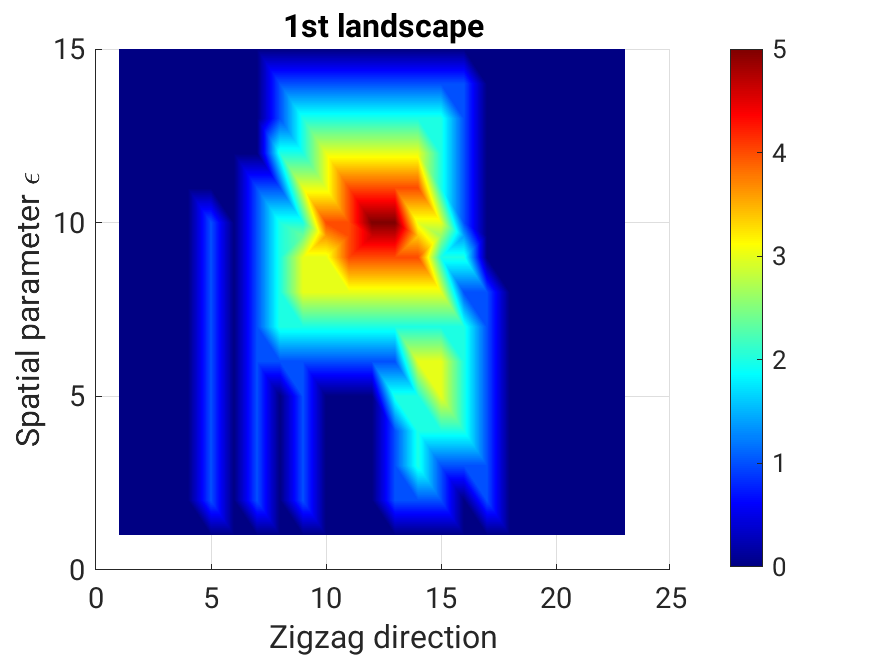}\;
	\includegraphics[width=0.49\textwidth]{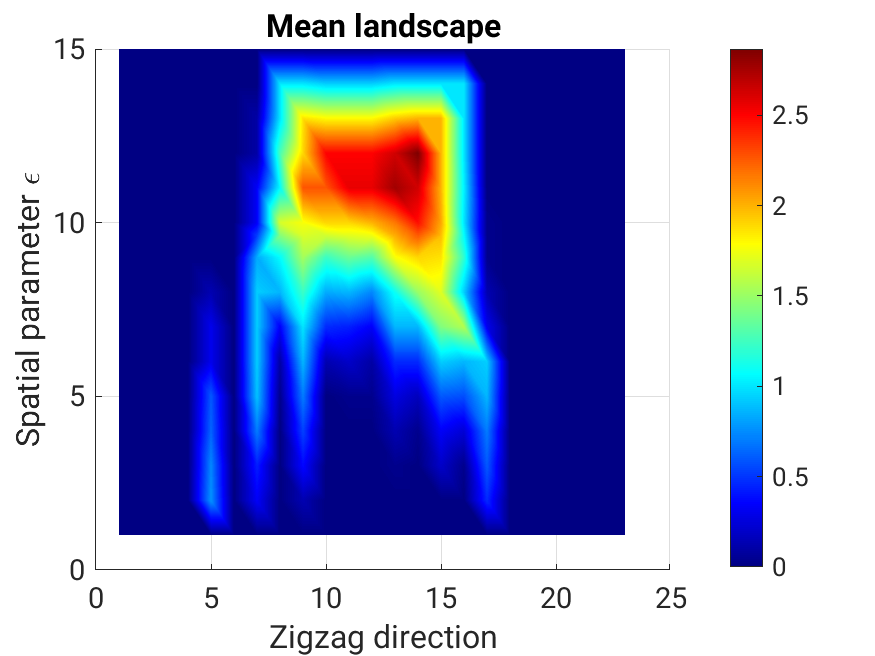}\\
	\includegraphics[width=0.49\textwidth]{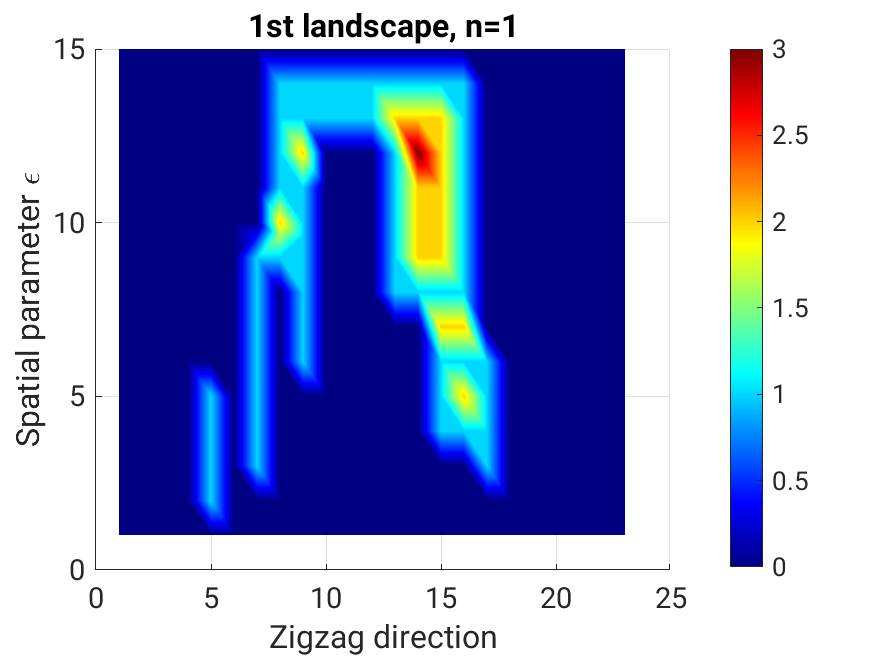}\;
	\includegraphics[width=0.49\textwidth]{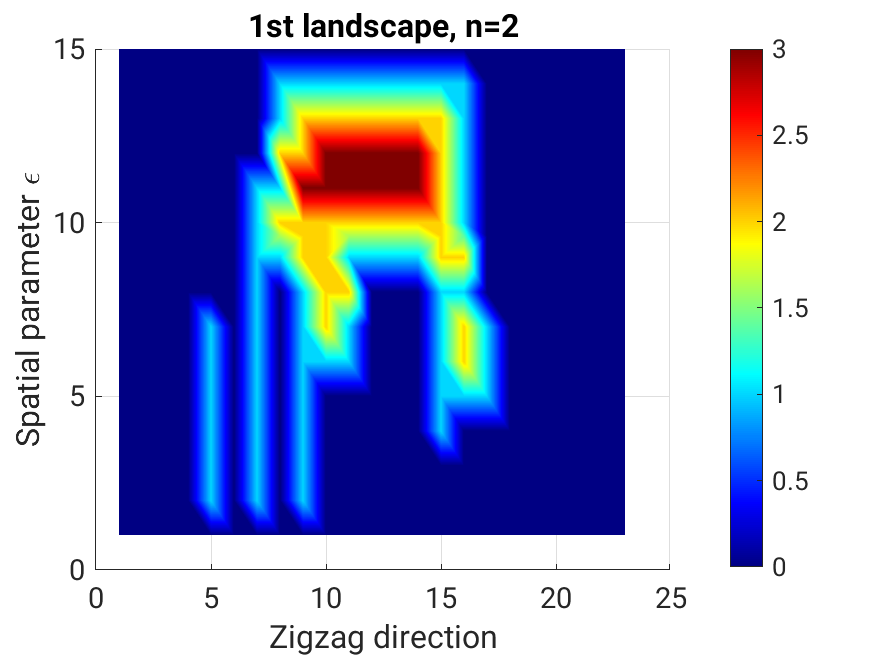}\\
	\includegraphics[width=0.49\textwidth]{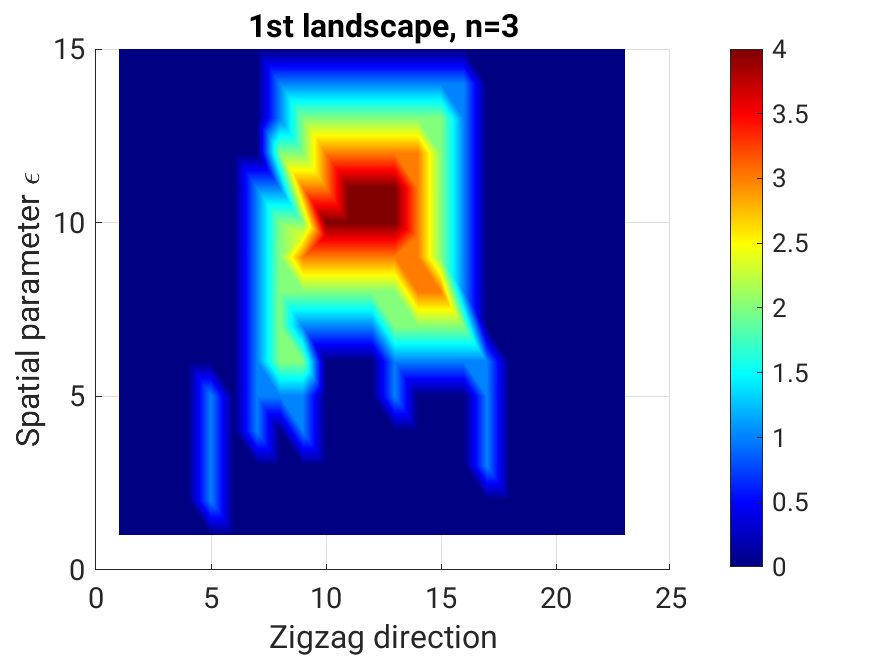}\;
	\includegraphics[width=0.49\textwidth]{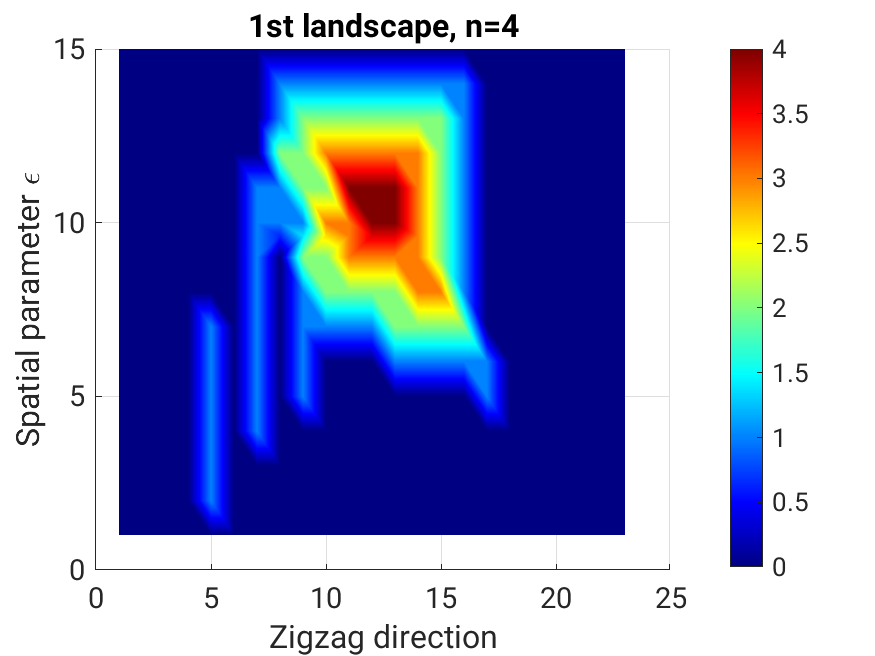}\\
	\includegraphics[width=0.49\textwidth]{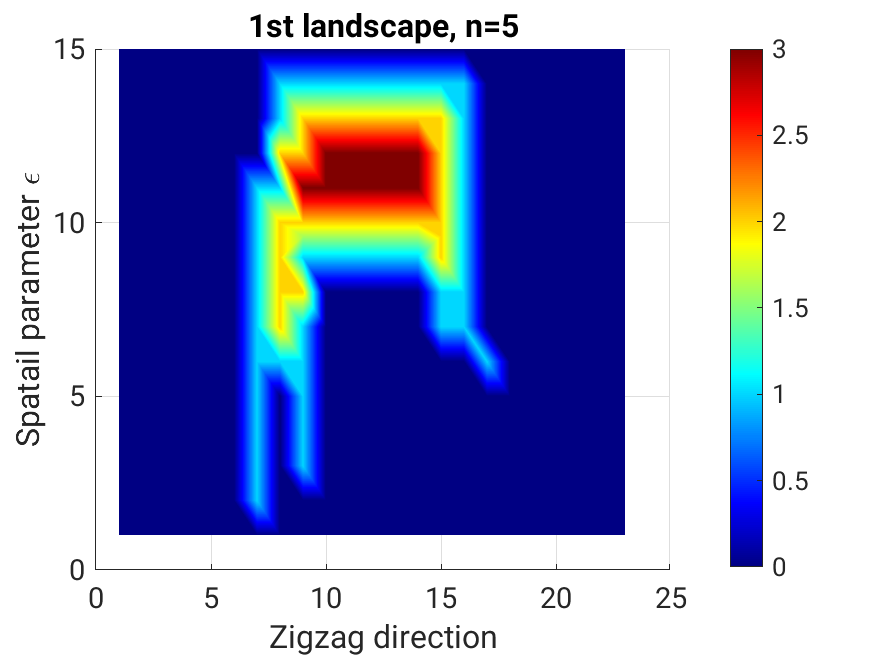}\;
	\includegraphics[width=0.49\textwidth]{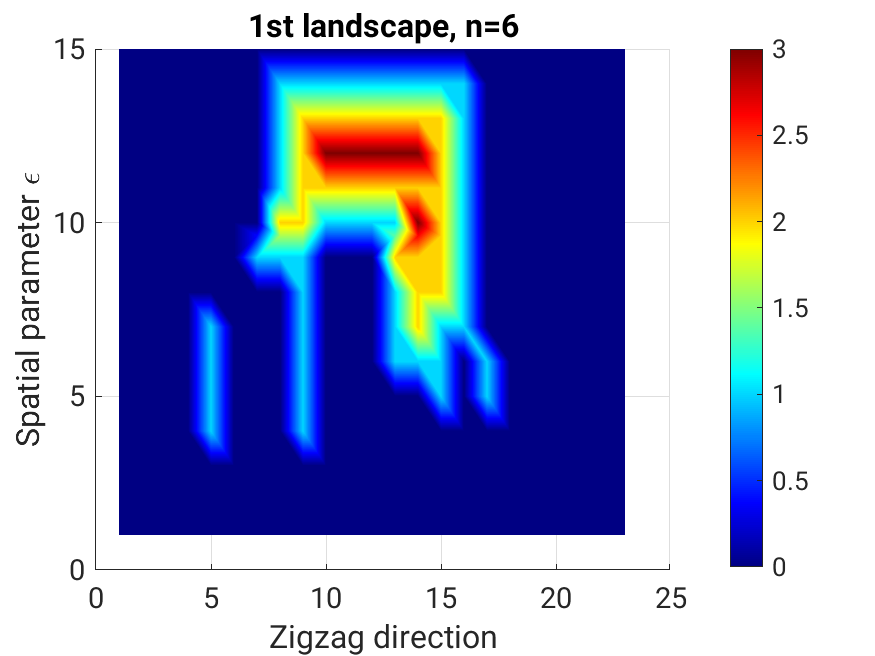}
	\caption{Top left: data without noise; top right: mean 1st landscape for 30 samplings; other rows: 1st landscapes for the first 6 noisy datasets}\label{fig:lands_selkov}
\end{figure}

\section{Discussion}\label{sec:discussion}
In this paper, we applied persistent homology to time series by regarding special kind of persistent modules which we call extended zigzag modules. The main advantage over the existing techniques is that one calculates features that are persistent simultaneously in space \textit{and} time direction by combining ideas from multiparameter and zigzag persistent homology. We proposed a way of visualization by defining spatiotemporal persistence landscapes for the extended zigzag modules. Furthermore, we define an interleaving distance for extended zigzag modules and proof stability of the spatiotemporal persistence landscapes with respect to the interleaving distance. To summarize, we defined a stable invariant taking values in a Banach space that carries useful statistical properties and can be used as an input for machine learning algorithms. \\
Since the behavior of the homological features can be observed over several spatial scales, one can utilize the landscapes to detect a suitable spatial scale to compute the ordinary zigzag persistent homology. In other applications, where zigzag homology was applied to time series \cite{tymochko2020}, it was a priori not clear how to choose a good spatial scale on which the persistence of features in time was observed. Our spatiotemporal persistence landscapes overcome this issue. \\
However, currently one disadvantage is the large computational cost and the high memory consumption. For every index in the parameter space one has to compute the persistence landscape as the barcode of a zigzag module through the parameter space. The application of techniques to reduce the size of the simplicial complex is expected to speed up the calculations as well as to reduce the storage needed.

\subsection*{Acknowledgments}
This work has been supported by the German Federal Ministry of Education and Research (BMBF-Projekt 05M20WBA und 05M20WWA: Verbundprojekt 05M2020 - DyCA). The first author (M.F.) thanks Michael Kerber (Graz University of Technology, Austria) for useful discussions.

\appendix	
	\section{Categorical definitions}\label{app:cat_defis}
	In this section, we revise some necessary categorical definitions. \\
	Let $\mathcal{C}$ be a category.
	\subsection{Limits and Colimits}
	\begin{definition}
		A \emph{diagram} $F$ indexed by a poset $(P,\leq)$ is a functor from the poset category $P$ to $\mathcal{C}$, i.e. every $p\in P$ is mapped to an object $F_p$ in $\mathcal{C}$ and any $p\leq q$ to a morphism $\phi_{p,q}:F_p\to F_q$ in $\mathcal{C}$, such that for any $p\leq q$ and $q\leq r$, it holds that $\phi_{q,r}\circ\phi_{p,q}=\phi_{p,r}$. 
	\end{definition}
	\begin{definition} Let $F:(P,\leq)\to\mathcal{C}$ be a diagram indexed by $(P,\leq)$. 
		A \emph{cone} of $F$ is an object $A$ of $\mathcal{C}$ together with a family $(\rho_p)_{p\in P}$ of morphisms $\rho_p:A\to F_p$, such that for any morphism $\phi_{p,q}:F_p\to F_q$ we have that $\rho_q=\phi_{p,q}\circ\rho_p$.\\
		A \emph{limit} of the diagram $F$ is a universal cone in the following sense: it is a cone $(L,(\psi_p)_{p\in P})$ with the property that for any other cone $(A,(\rho_p)_{p\in P})$ there exist a unique morphism $u:A\to L$ such that $\rho_p=\psi_p\circ u$ for all $p\in P$. 
	\end{definition}
	
	\begin{definition}
		Let $F:(P,\leq)\to\mathcal{C}$ be a diagram indexed by $(P,\leq)$. 
		A \emph{cocone} of $F$ is an object $B$ of $\mathcal{C}$ together with a family $(\tau_p)_{p\in P}$ of morphisms $\tau_p: F_p\to B$, such that for any morphism $\phi_{p,q}:F_p\to F_q$ we have that $\tau_p=\tau_q\circ\phi_{p,q}$.\\
		A \emph{colimit} of the diagram $F$ is a universal cocone in the following sense: it is a cocone $(C,(\sigma_p)_{p\in P})$ with the property that for any other cone $(B,(\tau_p)_{p\in P})$ there exist a unique morphism $u:C\to B$ such that $\tau_p=u\circ\sigma_p $ for all $p\in P$. 
	\end{definition} 
	Note that the limit and colimit, respectively, are \textit{essentially unique}, meaning that they are only unique up to an unique isomorphism. 
	
	\subsection{Kan extensions}\label{app:subsec_Kan}
	In order to extend persistence modules indexed by a poset $P$ to persistence modules indexed by another poset $Q$, which is a superset of $P$, we use categorical concepts that are called \textit{Kan extensions}. We will directly apply it to the setting of posets and functors between posets, however, the general definitions can be found in introductory books on category theory, for example in \cite{maclane_categories}. \\
	For a given functor between two posets $F:A\to B$ and a given $b\in B$ we define the sets 
	\begin{align*}
		A(F\leq b):=\{a\in A:F(a)\leq b\} \quad \mbox{ and }\quad A(F\geq b):=\{a\in A:F(a)\geq b\}.
	\end{align*}
	Let $M:A\to\cat{Vec}$ be a persistence module, then the \textit{left} Kan extension along the functor $F:A\to B$ is a persistence module $\mathrm{Lan}_F(M):B\to\cat{Vec}$ defined by
	\begin{align*}
		\mathrm{Lan}_F(M)(b):=\colimit M|_{ A(F\leq b)}.
	\end{align*}
	The structure maps $\mathrm{Lan}_F(M)(b)\to \mathrm{Lan}_F(M)(b')$ for $b\leq b'$ are given by the universal property of colimits. Again by universality of colimits, a morphism $f:M\to N$ between persistence modules $M,N\in\cat{Vec}^A$ induces a morphism $\mathrm{Lan}_F(f):\mathrm{Lan}_F(M)\to\mathrm{Lan}_F(N)$, making the left Kan extension a functor.\\
	In an analogous way we define the \textit{right} Kan extension of a persistence module $M:A\to\cat{Vec}$ along a functor $F:A\to B$ as
	\begin{align*}
		\mathrm{Ran}_F(M)(b):=\limit M|_{A(F\geq b)}
	\end{align*}
	with structure maps $\mathrm{Ran}_F(M)(b)\to \mathrm{Ran}_F(M)(b')$ for $b\leq b'$ given by the universal property of limits. As in the previous case, due to the universality of limits the right Kan extension is functorial meaning that it sends a morphism between persistence modules $f:M\to N$ to an induced morphism $\mathrm{Ran}_F(f):\mathrm{Ran}_F(M)\to\mathrm{Ran}_F(N)$.
	
	\begin{remark}
		Kan extensions are useful for the definition of \textit{continuous extensions} of discrete persistence modules. Given a $\Z^n$-indexed module, one obtains a $\R^n$-indexed module for example by taking the left Kan extension of the $\Z^n$-indexed module along the inclusion functor $\iota:\Z^n\hookrightarrow \R^n$.
	\end{remark}

	\section{Proof of Theorem \ref{theo_main}}\label{app:proof_main}
	For the proof of Theorem \ref{theo_main}, we closely follow the proof of Theorem 3.12 in \cite{dey2022}. However, in the part where we proof \underline{$\psi_{M_{\partial I}}=g\circ\xi\circ f$}, we have to extend the proof to make the theorem hold for a slightly more general class of paths $\Gamma_{\partial I}$.
	\begin{proof} Let $L:=\Gamma_{\min}$ and $U:=\Gamma_{\max}$ be lower and upper fences of interval $I$, with $\Gamma_{\min}$ and $\Gamma_{\max}$ as in Subsection \ref{subsec:gen_rank}. 
		Let further $(\limit M|_I, (\pi^M_p)_{p\in I})$, $(\limit M_{\partial I}, (\pi^{\partial I}_p)_{p\in \partial I})$, $(\limit M|_L, (\pi^L_p)_{p\in L})$ and $(\limit M|_U, (\pi^U_p)_{p\in U})$ be the limits of $M|_I$, $M_{\partial I}, M|_L$ and $M|_U$, respectively, and $(\colimit M|_I, (i^M_p)_{p\in I})$, $(\colimit M_{\partial I}, (i^{\partial I}_p)_{p\in \partial I})$, $(\colimit M|_L, (i^L_p)_{p\in L})$ and $(\colimit M|_U, (i^U_p)_{p\in U})$ be the colimits of $M|_I$, $M_{\partial I}, M|_L$ and $M|_U$, respectively. Note that $L$ is a lower fence of $I$ and $U$ is an upper fence of $I$.
		Keeping in mind the commutative diagram \ref{commDiag1}, but restricted to $M|_I$, it holds
		\begin{equation}
			\begin{tikzcd}[ampersand replacement=\&]
				\limit M|_L\arrow[r, "\xi"]\arrow[d,"e"] \& \colimit M|_U\arrow[d,"h"]\\
				\limit M|_I  \arrow[r,"\psi_{M|_I}"] \& \colimit M|_I,
			\end{tikzcd}\label{comm_diag_1_app}
		\end{equation}
		where $e$ and $h$ are isomorphisms. We want to prove that the rank of $\xi$ equals the rank of $\psi_{M_{\partial I}}$. To achieve this we want to show that there exists a surjective linear map $f:\limit M_{\partial I}\to \limit M|_L$ and an injective linear map $g:\colimit M|_U\to \colimit M_{\partial I}$ such that $\psi_{M_{\partial I}}=g\circ \xi \circ f$. We define the map $f$ as the canonical section restriction $(v_q)_{q\in\partial I}\mapsto (v_q)_{q\in L}$. The map $g$ is defined as the canonical map $[v_q]\mapsto [v_q]$ for any $q\in U$ and $v_q\in M_q$, which is the universal map from the colimit to the cocone.\\
		\underline{$\psi_{M_{\partial I}}=g\circ\xi\circ f$:} Let $p\in L$ and $q\in U$ with $p\leq q$ in $I$ (by the definitions of lower and upper fences such a choice exists). Since $p$ and $q$ also lie in $\partial I$, which is a path, there is a path $\Gamma =(p,p_1,...,p_n,q)$ of elements in $\partial I$. Let further $(v_r)_{r\in\partial I}\in\limit M_{\partial I}$. By definition of $\psi_{M_{\partial I}}$ we have that $\psi_{M_{\partial I}}((v_r)_{r\in\partial I})= [v_q]_{\partial I}$, where $[v_q]_{\partial I}$ is the equivalence class of $v_q$ in $\colimit M_{\partial I}$. On the other hand, we have that 
		\begin{align*}
			g\circ\xi\circ f((v_r)_{r\in\partial I})&=g\circ h^{-1}\circ\psi_M \circ e\circ f((v_r)_{r\in\partial I})=g\circ h^{-1}\circ i^M_p \circ \pi^M_p \circ e\circ f((v_r)_{r\in\partial I})\\
			&=g\circ h^{-1}\circ i^M_p \circ \pi^L_p\circ f((v_r)_{r\in\partial I})=g\circ h^{-1}\circ i^M_p \circ \pi^{L}_p((v_r)_{r\in L})\\
			&=g\circ h^{-1}\circ i^M_p (v_p)=g\circ h^{-1}\circ i^M_q (M_{p\leq q}(v_p)).
		\end{align*}
		Now, we define $w_q:=M_{p\leq q}(v_p)$. Since $q\geq p$, there is a path $\Gamma'=(q,p,p_1,...,p_n,q) $ in $P$ and a section $(w_q,v_p,v_{p_1},...,v_{p_n},v_q)$ of $M$ along $\Gamma'$ and hence by Proposition \ref{prop_path} it holds that $[w_q]_M=[v_q]_M$ in the colimit of $M$. In other words, $i^M_q(w_q)=i^M_q(v_q)$. This yields
		\begin{align*}
			g\circ h^{-1}\circ i^M_q (M_{p\leq q}(v_p))&=g\circ h^{-1}\circ i^M_q (w_q)=g\circ h^{-1}\circ i^M_q (v_q)=g\circ i^U_q(v_q) \\
			&=i^{\partial I}_q(v_q) = [v_q]_{\partial I}.
		\end{align*}
		It remains to show that $f$ is surjective and $g$ is injective.\\
		\underline{Surjectivity of $f$:} Let $r':\limit M\to\limit M_{\partial I}$ be the canonical section restriction map $(v_r)_{r\in P}\mapsto (v_r)_{r\in\partial I}$. Then, the restriction $r:\limit M\to \limit M|_L$ can be seen as composition of two restrictions, i.e. $r=f\circ r'$. Since $r$ is the inverse of the isomorphism $e$, $r$ is surjective and so is $f$.\\
		\underline{Injectivity of $g$:} Let $h':\colimit M_{\partial I}\to \colimit M$ be the unique map such that $i^{M}_q(v_q)=h'\circ i^{\partial I}_q(v_q)$ for all $q\in\partial I$ and $v_q\in M_q$, which exists by the universal property of $\colimit M_{\partial I}$ since $\colimit M$ is in particular a cocone of the diagram $M_{\partial I}$. Hence, $h'$ maps $[v_q]_{\partial I}$ to $[v_q]_{M}$. It holds that $h=h'\circ g$ for the isomorphism $h$ in Diagram \ref{comm_diag_1_app}. This shows that $g$ is injective.
	\end{proof}
	
	\section{Probability in Banach spaces}\label{app:prob_banach}
	
	Here we summarize a few results from the theory of probability of Banach spaces, as presented in \cite{bubenik2015}. We assume that $\mathcal{B}$ is a real separable Banach space with norm $\|\cdot\|$ and topological dual space $\mathcal{B}^*$. Assume further that $V:(\Omega,\mathcal{F},P)\to \mathcal{B}$ is a Borel measurable random variable defined on the probability space $(\Omega,\mathcal{F},P)$. By composition $\|V\|:\Omega\stackrel{V}{\to} \mathcal{B}\stackrel{\|\cdot\|}{\to}\mathbb{R}$ we obtain a real valued random variable. Furthermore, we obtain a real valued random variable by composition with elements $f$ of the dual space $f(V):\Omega\stackrel{V}{\to}\mathcal{B}\stackrel{f}{\to}\mathbb{R}$. Recall that the expected value of a real random variable $X:(\Omega,\mathcal{F},P)\to \R$ is defined as $E(X)=\int XdP=\int_{\Omega}X(\omega)dP(\omega)$. For random variables with values in a Banach space the so-called \textit{Pettis integral} yields an analogue to the expected value.
	
	\begin{definition}\label{def:pettis_integral}
		Let $V:(\Omega,\mathcal{F},P)\to \mathcal{B}$ be a Borel random variable with values in the real separable Banach space $\mathcal{B}$. An element $E(V)\in\mathcal{B}$ is called the Pettis integral of $V$ if $E(f(V))=f(E(V))$ for all $f\in\mathcal{B}^*$. 
	\end{definition}
	The following proposition yields a sufficient condition for Pettis integrability. 
	\begin{proposition}
		If $E\|V\|<\infty$, then $V$ has a Pettis integral and $\|E(V)\|\leq E\|V\|$.
	\end{proposition}
	For the further framework two notions of convergence of random variables with values in a Banach space are important. Let $(V_n)_{n\in \mathbb{N}}$ be a sequence of independent copies of $V$ and let further $S_n:=\sum_{i=1}^n V_n$. Analogously as for real valued random variables, we define that $(V_n)_{n\in\mathbb{N}}$ converges \textit{almost surely} to $V$ if $P(\lim_{n\to\infty}V_n=V)=1$. Furthermore, $(V_n)_{n\in\mathbb{N}}$ converges \textit{weakly} to $V$ if for all bounded continuous functions $\phi:\mathcal{B}\to\R$ it holds that $\lim_{n\to\infty}E(\phi(V_n))=E(\phi(V))$. 
	\begin{theorem}(Strong law of large numbers)\\
		It holds that $(\frac{1}{n}S_n)\to E(V)$ almost surely iff $E\|V\|<\infty$.
	\end{theorem}
	We call a random variable $V$ with values in a Banach space \textit{Gaussian} if $f(V)$ is a real Gaussian random variable with mean zero for each $f\in\mathcal{B}^*$. The \textit{covariance structure} of a random variable with values in a Banach space is defined as the set of expectations $E[(f(V)-E(f(V)))(g(V)-E(g(V)))]$ for $f,g\in\mathcal{B}^*$ and it determines a Gaussian random variable completely. 
	\begin{theorem}(Central limit theorem)\\
		Let $\mathcal{B}$ be a Banach space that has type 2. Let further $E(V)=0$ and $E(\|V\|^2)<\infty$. Then, $\frac{1}{\sqrt{n}}S_n$ converges weakly to a Gaussian random variable with the same covariance structure as $V$.
	\end{theorem}  
	Recall that for $2\leq p<\infty$, the $L^p$-spaces are of type 2.

\end{document}